\documentclass{amsart}

\usepackage{amssymb}
\usepackage[hidelinks]{hyperref}
\usepackage{todonotes}

\theoremstyle{plain}
\newtheorem{theorem}{Theorem}[section]
\newtheorem{lemma}[theorem]{Lemma}
\newtheorem{corollary}[theorem]{Corollary}
\newtheorem{proposition}[theorem]{Proposition}

\theoremstyle{definition}
\newtheorem{definition}[theorem]{Definition}

\newtheorem{question}[theorem]{Question}

\theoremstyle{remark}
\newtheorem*{remark}{Remark}

\newcommand{\bP}{\mathbb{P}}

\newcommand{\bQ}{\mathbb{Q}}
\newcommand{\Q}{\bQ}

\newcommand{\cA}{\mathcal{A}}
\newcommand{\cB}{\mathcal{B}}

\newcommand{\cD}{\mathcal{D}}

\newcommand{\cF}{\mathcal{F}}
\newcommand{\cG}{\mathcal{G}}

\newcommand{\cI}{\mathcal{I}}
\newcommand{\I}{\cI}
\newcommand{\cJ}{\mathcal{J}}
\newcommand{\J}{\cJ}
\newcommand{\cK}{\mathcal{K}}

\newcommand{\cM}{\mathcal{M}}

\newcommand{\cP}{\mathcal{P}}

\newcommand{\cT}{\mathcal{T}}

\newcommand{\cU}{\mathcal{U}}
\newcommand{\cW}{\mathcal{W}}

\newcommand{\continuum}{\mathfrak{c}} 
 
\newcommand{\pnumber}{\mathfrak{p}} 
\newcommand{\cc}{\mathfrak{c}}

\newcommand{\uu}{\mathfrak{u}}
\newcommand{\pp}{\mathfrak{p}}
\DeclareMathOperator{\add}{add}
\DeclareMathOperator{\cov}{cov}
\DeclareMathOperator{\adds}{\add^*}
\DeclareMathOperator{\wsp}{\mathfrak{fip}}

\DeclareMathOperator{\FinBW}{FinBW}
\DeclareMathOperator{\Kat}{K}
\DeclareMathOperator{\KB}{KB}
\DeclareMathOperator{\point}{P}

\newcommand{\fin}{\mathrm{Fin}}
\newcommand{\Exh}{\mathrm{Exh}} 
\newcommand{\Fin}{\mathrm{Fin}}
\newcommand{\ED}{\mathcal{ED}} 
\newcommand{\conv}{\mathrm{conv}} 
 
\newcommand{\nwd}{\mathrm{nwd}} 
\newcommand{\lacunary}{\mathcal{L}} 
\newcommand{\random}{\mathcal{R}}

\newcommand{\poset}{\bP} 

\begin{document}


\title{Characterizing existence of certain ultrafilters}


\address{Institute of Mathematics\\ Faculty of Mathematics, Physics and Informatics\\ University of Gda\'{n}sk\\ ul.~Wita Stwosza 57\\ 80-308 Gda\'{n}sk\\ Poland}

\author[R.~Filip\'{o}w]{Rafa\l{} Filip\'{o}w}
\email{Rafal.Filipow@ug.edu.pl}
\urladdr{http://mat.ug.edu.pl/~rfilipow}

\author[K.~Kowitz]{Krzysztof Kowitz}
\email{Krzysztof.Kowitz@phdstud.ug.edu.pl}

\author[A.~Kwela]{Adam Kwela}
\email{Adam.Kwela@ug.edu.pl}
\urladdr{http://kwela.strony.ug.edu.pl/}


\date{\today}






\begin{abstract}
Following Baumgartner [J. Symb.~Log.~60 (1995), no.~2], for an ideal $\I$ on $\omega$, we say that an ultrafilter $\cU$ on $\omega$  is 
an \emph{$\I$-ultrafilter} if for every function $f:\omega\to\omega$ there is $A\in \cU$ with $f[A]\in \I$.
		 
If there is an $\I$-ultrafilter which is not a $\J$-ultrafilter, then $\I$ is not below $\J$ in the Kat\v{e}tov order $\leq_{\Kat}$ (i.e.~for every function $f:\omega\to\omega$ there is $A\in \cI$ with $f^{-1}[A]\notin \J$). On the other hand, in general $\I\not\leq_{\Kat}\J$ does not imply that existence of an $\I$-ultrafilter which is not a $\J$-ultrafilter is consistent.
 
We provide some sufficient conditions on ideals to obtain the equivalence: $\I\not\leq_{\Kat}\J$ if and only if it is consistent that there exists an $\I$-ultrafilter which is not a $\J$-ultrafilter. In some cases when the Kat\v{e}tov order is not enough for the above equivalence, we provide other conditions for which a similar equivalence holds. We are mainly interested in the cases when the family of all $\I$-ultrafilters or $\J$-ultrafilters coincides with some known family of ultrafilters: P-points, Q-points or selective ultrafilters (a.k.a.~Ramsey ultrafilters). In particular, our results provide a characterization of Borel ideals $\I$ which can be used to characterize P-points as $\I$-ultrafilters. 

Moreover, we introduce a cardinal invariant which is used to obtain a sufficient condition for the existence 
of an $\I$-ultrafilter which is not a $\J$-ultrafilter. Finally, we prove some new results concerning existence of certain ultrafilters under various set-theoretic assumptions.
\end{abstract}


\maketitle




\section{Introduction}

All notions and notations used in the introduction are defined in Section~\ref{sec:preliminaries}.

If there is an $\I$-ultrafilter which is not a $\J$-ultrafilter, then $\I\not\leq_{\Kat}\J$. On the other hand, in general $\I\not\leq_{\Kat}\J$ does not imply that existence of an $\I$-ultrafilter which is not a $\J$-ultrafilter is consistent. For instance,  $\fin^2\not\leq_{\Kat}\conv$,  but $\fin^2$-ultrafilters and $\conv$-ultrafilters coincide. 
We provide some sufficient conditions on ideals to obtain equivalence: $\I\not\leq_{\Kat}\J$ if and only if it is consistent that there exists an $\I$-ultrafilter which is not a $\J$-ultrafilter. In some cases when the Kat\v{e}tov order is not enough for the above equivalence, we provide other conditions for which a similar equivalence holds.  
We are mainly interested in the cases when the family of all $\I$-ultrafilters or $\J$-ultrafilters coincides with some known family of ultrafilters: P-points, Q-points or selective ultrafilters.

In a series of papers \cite{MR2798896,MR2194039,MR2512901,MR2752957}, Fla\v{s}kov\'{a} constructed $\I$-ultrafilters which are not $\J$-ultrafilters for various pairs of ideals $\I$ and $\J$ under various set-theoretic assumptions. A closer analysis of her constructions allows us to define a combinatorial ``core'' of her proofs. Namely, we define a new cardinal characteristic 
and show how to use it to prove existence of $\I$-ultrafilters which are not $\J$-ultrafilters. We prove some new results of this type.

It is known that some well known families of ultrafilters can be characterized as $\I$-ultrafilters with the aid of definable (usually Borel) ideals $\I$. For instance, an ultrafilter  is a P-point if and only if it is a $\fin^2$-ultrafilter. 
The ideal $\fin^2$ is Borel, in fact it is $F_{\sigma\delta\sigma}$ and it is known that P-points cannot be characterized as $\I$-ultrafilters for any $F_\sigma$ ideal $\I$ (see \cite[Theorem~3.1]{MR2512901}).
Our results provide a characterization of Borel ideals $\I$ which can be used to characterize P-points as $\I$-ultrafilters.


The paper is organized in the following way.
In Section~\ref{sec:preliminaries}, we review the notions and notations used in the rest of the paper. We also provide some basic facts needed in the sequel, and we prove those of them that we could not  find in the literature.  
In Section~\ref{sec:cardinal-nad-order} we introduce a cardinal invariant and an order which provide us with the main tools that we use in the rest of the paper. Moreover, we show how those tools can be used to construct specific $\I$-ultrafilters and characterize existence of certain classes of $\I$-ultrafilters.
In Section~\ref{sec:lower-bound} we find a lower bound for the above mentioned invariant in terms of some other cardinal characteristics known from the literature.
Moreover, we discover a class of ideals for which the following equivalence holds: $\I\not\leq_{\Kat}\J$ if and only if it is consistent that there exists an $\I$-ultrafilter which is not a $\J$-ultrafilter. 
In Section~\ref{sec:game} we introduce a game that is designed to provide one more lower bound for our cardinal invariant in the case when $\I$ contains a tall summable ideal. This enables us to relax the set-theoretic assumptions in some of our results. 
The remaining sections of the paper are devoted to three well know classes of ultrafilters: P-points, Q-points and selective ultrafilters.
Among others, we characterize existence of P-points (Q-points, selective ultrafilters, resp.) which are not $\J$-ultrafilters. 
It turns out, that in the case of Q-points this characterization can be expressed solely in terms of the Kat\v{e}tov order. On the other hand, in the case of P-points (selective ultrafilters, resp.), the Kat\v{e}tov order is not enough to obtain such a characterization. However, we managed to show that extendability to $P^+$-ideals (selective ideals, resp.) are enough in these cases. The characterization mentioned in the previous paragraph is stated in Section 6. Finally, in our last Section~\ref{sec:By-products} we present some by-products of our studies: we show that every $G_{\delta\sigma}$ ideal is extendable to an $F_{\sigma}$ ideal and that neither the ideals $\lacunary$ (generated by lacunary sets) nor the ideal $\random$ (generated by homogeneous sets in some fixed instance of the random graph on $\omega$) is $\leq_{\Kat}$-homogeneous.


\section{Preliminaries}
\label{sec:preliminaries}

All the notions and notations used in the paper, but not defined in this section seem to be quite standard and can be found for instance in \cite{MR1940513,MR1321597,MR2768685}.

By $\omega$ we denote the set of all natural numbers.
We identify  a natural number $n$ with the set $\{0, 1,\dots , n-1\}$ (for instance, $n\setminus k$ means the set $\{i\in\omega: k\leq i<n\}$).


\subsection{Ideals}

A  family $\I\subseteq\cP(X)$ is called an \emph{ideal on $X$} if it satisfies the following conditions:
\begin{enumerate}
	\item if $A,B\in \I$, then $A\cup B\in\I$,
	\item if $A\subseteq B$ and $B\in\I$, then $A\in\I$,
	\item $\I$ contains all finite subsets of $X$,
	\item $X\notin\I$.
\end{enumerate}
A  family $\cF\subseteq\cP(X)$ is called a \emph{filter on $X$} if it satisfies the following conditions:
\begin{enumerate}
	\item if $A,B\in \cF$, then $A\cap B\in\cF$,
	\item if $A\subseteq B$ and $A\in\cF$, then $B\in\cF$,
	\item $\cF$ contains all cofinite subsets of $X$,
	\item $\emptyset\notin\cF$.
\end{enumerate}
If $\cA\subseteq\cP(X)$ then we write $\cA^*=\{B\subseteq X: X\setminus B\in\cA\}$. It is easy to see that if $\I$ is an ideal then $\I^*$ is a filter (called the \emph{dual filter of $\I$}). Conversely, if $\cF$ is a filter then $\cF^*$ is an ideal (called the \emph{dual ideal of $\cF$}). 

We say that an ideal $\I$ on $X$ is generated by a family $\cA\subseteq\cP(X)$ whenever $\I=\{A\subseteq X:\exists F\in[\cA]^{<\omega}(A\subseteq\bigcup F)\}$. For an ideal $\I$ on $X$, 
we write $\I^+=\{A\subseteq X: A\notin\I\}$ and call it the \emph{coideal of $\I$}.
The ideal of all finite subsets of an infinite set $X$
is denoted by $\fin(X)$ (we write $\fin$ instead of $\fin(\omega)$ for short).
For an ideal $\I$ on $X$ and $A\subseteq X$ we define $\I\restriction A = \{B\cap A: B\in \I\}$. It is easy to see that $\I\restriction A$ is an ideal on $A$ if and only if $A\notin\I$.
An ideal $\I$ on $X$ is
\emph{tall} if for every infinite $A\subseteq X$ there is an infinite $B\in\I$ such that $B\subseteq A$ (some authors use the name \emph{dense ideal} in this case).
An ideal $\I$ on $X$ is a \emph{P-ideal} (\emph{weak P-ideal}, resp.) if for any countable family $\cA\subseteq\I$ there is $B\in \I^*$ ($B\in \I^{+}$, resp.) such that $A\cap B$ is finite  for every $A\in \cA$. 
An ideal $\I$ on $X$ is a \emph{$P^+$-ideal} (\emph{selective ideal}, resp.) if for any decreasing sequence $\{A_n:n\in\omega\}\subseteq  \I^+$ there is $A \in \I^+$, $A\subseteq A_0$ such that
$A\cap (A_n\setminus A_{n+1})$ is finite  
($|A\cap (A_n\setminus A_{n+1})|\leq 1$, resp.)
for every $n$.

For $A\subseteq \omega\times \omega$ and $n\in \omega$,
we write 
$A_{(n)}=\{k\in \omega: (n,k)\in A\}$
i.e.~$A_{(n)}$ is 
the vertical 
section of $A$ at the point $n$.

In the sequel, we will use the following ideals.
\begin{itemize}
\item
	$\fin^2 = \{A\subseteq\omega\times\omega:\forall^\infty n \, (|A_{(n)}|<\omega)\}$.
	\item $\ED=\{A\subseteq\omega\times\omega:\exists m  \, \forall^\infty n \, (|A_{(n)}|<m)\}$.
	\item $\ED_{\fin}=\{A\subseteq\Delta:\exists m  \, \forall n \, (|A_{(n)}|<m)\}$, where $\Delta = \{(n,k)\in\omega^2: n\geq k\}$.
\item The \emph{summable ideals}: 
$\I_g = \{A\subseteq\omega: \sum_{n\in A}g(n)<\infty\}$, where $g:\omega\to[0,\infty)$ satisfies $\sum_{n\in \omega}g(n)=\infty$.
\item $\nwd = \{A\subseteq \Q\cap [0,1]: \text{$A$ is nowhere dense}\}$.
\item $\conv = \{A\subseteq \Q\cap [0,1]: \text{$A$ has at most finitely many limit points}\}$.
\item $\random$ is the ideal generated by homogeneous sets in some fixed instance of the random graph on $\omega$ (see e.g.~\cite[p.29]{alcantara-phd-thesis} for details).
\item $\cW$ is the ideal of all subsets of $\omega$ which does not contain arbitrary long finite arithmetic progression.
\end{itemize}

Let $\I,\J$ be ideals on $X$ and $Y$ respectively.
We write $\J\approx \I$ if there is a bijection $f:X\to Y$ such that $f^{-1}[B]\in \I \iff B\in \J$, and say that $\I$ and $\J$ are \emph{isomorphic}.

One can think about all of the above ideals as ideals on $\omega$ by identifying the domain set of the ideal (for instance, $\omega\times\omega$ or $\Q\cap[0,1]$) with $\omega$ via a fixed bijection.

If $\I\not\approx\fin$ and $\I \not\approx \fin\oplus\cP(\omega)=\{A\subseteq \omega\times 2:(\omega\times \{0\})\cap A\text{ is finite}\}$, then  we define 
$\adds(\I)  = \min\{|\cA|:\cA\subseteq \I \land 
\neg(\exists B\in \I\, \forall A\in \cA\,(|A\setminus B|<\omega)\}$.
It is easy to see that $\adds(\I) = \omega$  for non P-ideals and $\adds(\I)\geq \omega_1$ for P-ideals. 


\subsection{Borel ideals}

An ideal on $\omega$ is called \emph{$F_\sigma$} ($\mathbf{\Sigma}^0_\alpha$,  Borel, analytic and so on, resp.) if it is an $F_\sigma$ ($\mathbf{\Sigma}^0_\alpha$, Borel, analytic and so on, resp.) subset of $\cP(\omega)$ with a topology induced from the Cantor space $\{0,1\}^\omega$ by identifying subsets of $\omega$ with their characteristic functions.

A map $\phi: \cP(\omega)\to [0,+\infty]$ is a \emph{submeasure on $\omega$} if for every $A,B\subseteq \omega$ we have: $\phi(\emptyset)=0$ and $\phi(A)\leq \phi(A\cup B) \leq \phi(A)+\phi(B)$. A submeasure $\phi$ is \emph{lower semicontinuous} (in short: lsc) if $\phi(A)=\lim_{n\to\infty}\phi(A\cap n)$ for all $A\subseteq \omega$. For any lsc submeasure $\phi$, we define the following families: $\Fin(\phi)=\{A\subseteq \omega: \phi(A)<\infty\}$ and $\Exh(\phi)=\{A\subseteq \omega: \lim_{n\to \infty}\phi(A\setminus n)=0\}$.

\begin{theorem}[\cite{MR1124539} and \cite{MR1708146}]
\label{lscsm}
\ 
\begin{enumerate}
\item If $\I$ is an $F_{\sigma}$-ideal, then there is a lsc submeasure $\phi$ such that $\I=\Fin(\phi)$.
\item If $\I$ is an analytic P-ideal, then there is a lsc submeasure $\phi$ such that $\I=\Exh(\phi)$.
\item If  $\I$ is an $F_{\sigma}$ P-ideal, then there is a lsc submeasure $\phi$ such that $\I=\Exh(\phi)=\Fin(\phi)$.
\end{enumerate}
\end{theorem}

It is not difficult to see that each summable ideal $\I_g$ is an $F_\sigma$ P-ideal, and it is tall iff $\lim_{n\to \infty}g(n)=0$.

\begin{proposition}[{\cite[Proposition~4.5]{MR2861027}}]
	\label{prop:tall-analytic-P-ideal-contains-summable-ideal}
	Each tall analytic P-ideal contains a tall summable ideal. In particular, an ideal contains a tall summable ideal if and only if it contains a tall analytic P-ideal.
\end{proposition}


\subsection{Kat\v{e}tov order}

By $Y^X$ we denote the family of all functions from $X$ into $Y$.
A function $f\in Y^X$ is \emph{finite-to-one} if the inverse image $f^{-1}[\{y\}]$  is finite for every $y\in Y$.

Let $\I,\J$ be ideals on $X$ and $Y$ respectively.
	For $\cF\subseteq Y^X$, we write 
$\J\leq_{\cF}\I$ if there is a function $f\in\cF$ such that $f^{-1}[B]\in \I$ for each $B\in \J$. In particular, we write
	\begin{enumerate}
		\item 
$\J\leq_{\Kat}\I$ if $\cF=Y^X$ ($\leq_{\Kat}$ is called the \emph{Kat\v{e}tov order}),

		\item 
$\J\leq_{\KB}\I$ if $\cF$ is the family of all finite-to-one functions from $X$ to $Y$ ($\leq_{\KB}$ is called the \emph{Kat\v{e}tov-Blass order}),

		\item 
$\J\leq_{\point}\I$ if $\cF$ is the family of all one-to-one functions from $X$ to $Y$,
		\item 
$\J\sqsubseteq\I$ if $\cF$ is the family of all bijections from $X$ to $Y$ (in this case we say that $\I$ contains an isomorphic copy of $\J$).
\end{enumerate}

In the sequel, we will be particularly interested in three families of functions: 
\begin{itemize}
	\item 
$\omega^\omega$ -- the family  of all functions $f:\omega\to\omega$
\item $\fin-1$ -- the family of all finite-to-one functions $f:\omega\to\omega$,
\item $1-1$ -- the family of all one-to-one functions  $f:\omega\to\omega$. 
\end{itemize}
Moreover, as we can think about ideals on countable sets as ideals on $\omega$, we will sometimes write $\I\leq_{\omega^\omega}\J$, $\I\leq_{\fin-1}\J$ or $\I\leq_{1-1}\J$ even if $\I$ and $\J$ are defined on other countable sets than $\omega$.

\begin{proposition}\ 
\label{prop:ideals-not-below-abov-other-ideals-in-Katetov-order}
\label{prop:selective-ideal-not-K-above-ED}
\label{prop:P-ideal-not-K-below-FINxFIN}
\label{prop:P-ideal-not-K-above-FINxFIN}
\label{prop:P-plus-ideal-not-K-above-FINxFIN}
\label{prop:ED-FIN-not-K-below-FINxFIN}
\label{prop:P-ideal-not-K-below-ED-FIN}
\begin{enumerate}
	\item 
$\ED_{\fin}\not\leq_{\Kat}\fin^2$.\label{prop:ED-FIN-not-K-below-FINxFIN:item}

	\item 
$\fin^2\not\leq_{\Kat}\I$ for each P-ideal $\I$.\label{prop:P-ideal-not-K-above-FINxFIN:item}

\item 
$\fin^2\not\leq_{\Kat} \I$ for each $P^+$-ideal $\I$.\label{prop:P-plus-ideal-not-K-above-FINxFIN:item}

\item 
$\ED\not\leq_{\Kat} \I$ for each selective ideal $\I$.\label{prop:selective-ideal-not-K-above-ED:item}

\item 
	$\I\not\leq_{\Kat}\fin^2$ for each tall P-ideal $\I$.\label{prop:P-ideal-not-K-below-FINxFIN:item}

\item 
$\I\not\leq_{\Kat}\ED_\fin$ for each tall P-ideal $\I$.\label{prop:P-ideal-not-K-below-ED-FIN:item}

\end{enumerate}
\end{proposition}

\begin{proof}
(\ref{prop:ED-FIN-not-K-below-FINxFIN:item})
See \cite{MR3696069}.
(\ref{prop:P-ideal-not-K-above-FINxFIN:item}) See \cite[Observation~2.3]{MR3600759}
(\ref{prop:P-plus-ideal-not-K-above-FINxFIN:item}) See \cite[Corollary 6.7]{MR3868039}.

(\ref{prop:selective-ideal-not-K-above-ED:item})
Take any $f:\omega\to\omega\times\omega$
and define
$A_n = f^{-1}[(\omega\setminus n )\times \omega]$ for each $n$.
If $A_n\in\I$ for some $n$, then $f^{-1}[n\times\omega]\notin\I$ and the proof is finished.
Suppose that $A_n\notin\I$ for each $n$.
Since $\I$ is selective, there is $A\notin\I$ such that $|A\cap (A_n\setminus A_{n+1})|\leq 1$ for each $n$.
Then $B=f[A]\in \ED$ and $f^{-1}[B]\notin \I$, so the proof is finished.

(\ref{prop:P-ideal-not-K-below-FINxFIN:item})
	Suppose  that $\I\leq_{\Kat}\fin^2$.
	Let $f:\omega\times\omega\to\omega$ be such that 
	$f^{-1}[B]\in \fin^2$ for each $B\in \I$.
	We have 2 cases.
	
	Case (1).
	The set $f[\{n\}\times \omega]$ is finite for infinitely many $n$.
	Let $C =\{n\in \omega: f[\{n\}\times \omega] \text{\ is finite}\}$. 
	Then for each $n\in C$ there is $k_n\in f[\{n\}\times \omega]$ such that 
	$(f^{-1}[\{k_n\}])_{(n)}$ is infinite.
	Let $A=\{k_0,k_1,\dots\}$.
	Since  $f^{-1}[A]\notin \fin^2$, we have $A\notin\I$. In particular, $A$ has to be infinite.
	Since $\I$ is tall, there is an infinite set $B\in \I$ with $B\subseteq A$.
	But then $f^{-1}[B]\notin\fin^2$, a contradiction.
	
	Case (2).
	The set $f[\{n\}\times \omega]$ is infinite for all but finitely many $n$.
	Let $C =\{n\in \omega: f[\{n\}\times \omega] \text{\ is infinite}\}$. 
	Since $\I$ is tall, for each $n\in C$ there is an infinite set $B_n\in \I$, with $B_n\subseteq f[\{n\}\times \omega]$. Observe that the set $(f^{-1}[B_n])_{(n)}$ is infinite for each $n\in C$.
	
	Since $\I$ is a P-ideal, there is $B\in \I$ such that $B_n\setminus B$ is finite for each $n$.
	
	As $f^{-1}[B]\in \fin^2$, the set $(f^{-1}[B])_{(n)}$ is finite for all but finitely many $n$, say for all $n\in D$. 
	
	All in all, the set $(f^{-1}[B_n\setminus B])_{(n)} = (f^{-1}[B_n])_{(n)}\setminus (f^{-1}[B])_{(n)}$ is infinite for all $n\in C\cap D$.
	
	Since $B_n\setminus B$ is finite, for each $n\in C\cap D$ there is $c_n\in B_n\setminus B$ such that $(f^{-1}[\{c_n\}])_{(n)}$ is infinite.
	
	Let $A=\{c_n: n\in C\cap D\}$.
	Note that $C\cap D$ is infinite. Thus, $f^{-1}[A]\notin \fin^2$ and consequently $A\notin\I$. In particular, $A$ is infinite.
	Since $\I$ is tall, there is an infinite set $A'\in \I$ with $A'\subseteq A$.
	But then $f^{-1}[A']\notin\fin^2$, a contradiction.

(\ref{prop:P-ideal-not-K-below-ED-FIN:item})
Fix any $f:\Delta\to\omega$, where $\Delta=\{(i,j)\in\omega^2:\ j\leq i\}$ (i.e., $\ED_\fin=\ED\restriction \Delta$). Suppose to the contrary that $f$ witnesses $\I\leq_{\Kat}\ED_\fin$. 

Observe that for each $n\in\omega$ and an infinite $A\subseteq\omega$ one can find $k>n$ such that $f[(A\times\{k\})\cap\Delta]$ is infinite. Indeed, otherwise we can inductively pick sets $D_i\in[\omega]^\omega$, for $i>n$, such that for each $i>n$ we have:
\begin{itemize}
	\item $D_{i+1}\subseteq D_i$;
	\item $D_i\times \{i\}\subseteq\Delta$;
	\item $f[D_i \times \{i\}]=\{x_i\}$ for some $x_i\in\omega$.
\end{itemize}
Then either $\{x_i:\ i>n\}\in\fin\subseteq\I$ and $f^{-1}[\{x_i:\ i>n\}]\supseteq \bigcup_{i>n}D_i\times \{i\} \notin\ED_\fin$ or using tallness of $\I$ we could find an infinite $C\subseteq \{x_i:\ i>n\}$, $C\in\I$ such that 
$$f^{-1}[C]\supseteq \bigcup\{D_i \times \{i\} :\ x_i\in C\}\notin\ED_\fin.$$

We will inductively define sequences $\{k_n:n\in\omega\}\subseteq\omega$ and $\{E_n:n\in\omega\}\subseteq[\omega]^\omega$ such that for each $n\in\omega$:
\begin{itemize}
	\item $E_{n+1}\subseteq E_n$;
	\item $k_{n+1}>k_n$;
	\item $E_n\times\{k_n\}\subseteq\Delta$;
	\item $f\restriction E_n\times\{k_n\}$ is one-to-one;
	\item $f[E_n\times\{k_n\}]\in\I$.
\end{itemize}
Start by choosing any $k_0\in\omega$ such that $f[(\omega\times\{k_0\})\cap\Delta]$ is infinite (such $k_0$ exists by the above observation) and any infinite $E'_0\subseteq\omega$ such that $E'_0\times\{k_0\}\subseteq\Delta$ and $f\restriction E'_0\times\{k_0\}$ is one-to-one. As $\I$ is tall, there is $E_0\subseteq E'_0$ such that $f[E_0\times\{k_0\}]\in\I$. At step $n+1$, find any $k_{n+1}>k_n$ such that $f[E_n\times\{k_{n+1}\}\cap\Delta]$ is infinite (such $k_{n+1}$ exists by the observation from the second paragraph of this proof) and pick an infinite $E'_{n+1}\subseteq E_n$ such that $E'_{n+1}\times\{k_{n+1}\}\subseteq\Delta$ and $f\restriction E'_{n+1}\times\{k_{n+1}\}$ is one-to-one. Again, using tallness of $\I$ find $E_{n+1}\subseteq E'_{n+1}$ such that $f[E_{n+1}\times\{k_{n+1}\}]\in\I$. 

Once the induction is completed, as $\I$ is a P-ideal, there is $B\in\I$ such that $f[E_n\times\{k_n\}]\setminus B$ is finite for all $n\in\omega$. Then $F_n=E_n\setminus\{i\in E_n:\ f(i,k_n)\in B\}\in\fin$ as $f\restriction E_n\times\{k_n\}$ is one-to-one. Thus 
$$f^{-1}[B]\supseteq\bigcup_{n\in\omega}\left((E_n\setminus F_n)\times\{k_n\}\right)\notin\ED_\fin$$
as for each $m\in\omega$ we have 
$$\left|(\{l\}\times\omega)\cap\bigcup_{n\in\omega}\left((E_n\setminus F_n)\times\{k_n\}\right)\right|\geq m,$$
where $l=\min(E_m\setminus\bigcup_{j\leq m}F_j)$.
\end{proof}


\subsection{Ultrafilters and \textbf{$\I$}-ultrafilters}
\label{subsec:ultrafilters}

Recall that an ultrafilter $\cU$ on $\omega$  is 
\begin{enumerate}
	\item 
a \emph{P-point} if $\cU^*$ is a P-ideal (equivalently, if $\cU^*$ is a $P^+$-ideal),
	
	\item 
a \emph{selective ultrafilter} (a.k.a.~\emph{Ramsey ultrafilter}) if $\cU^*$ is a selective ideal,

	\item 
a \emph{Q-point} if for each partition $\{A_n:n\in\omega\}$ of $\omega$ into finite sets there is $A\in \cU$ such that $|A_n\cap A|=1$ for each $n$.
\end{enumerate}
It is known that an ultrafilter is selective if and only if it is both a P-point and a Q-point (see e.g.~\cite{MR2768685}).

Following Baumgartner, Brendle and Fla\v{s}kov\'{a}~\cite{MR1335140,MR3600759}, for an ideal $\I$ on $\omega$, we say that an ultrafilter $\cU$ on $\omega$  is 
\begin{enumerate}
\item an \emph{$\I$-ultrafilter} if $\I\not\leq_{\Kat} \cU^*$ i.e.~there is no function $f:\omega\to\omega$ such that $f^{-1}[B]\in \cU^*$ for each $B\in \I$,

\item 
a \emph{weak $\I$-ultrafilter} if $\I\not\leq_{\KB} \cU^*$ i.e.~there is no finite-to-one function $f:\omega\to\omega$ such that $f^{-1}[B]\in \cU^*$ for each $B\in \I$,

\item 
an \emph{$\I$-point} if $\I\not\leq_{\point} \cU^*$ i.e.~there is no one-to-one function $f:\omega\to\omega$ such that $f^{-1}[A]\in\cU^\star$ for all $A\in\I$.
If $\I$ is tall, then, using \cite[Lemma 3.3]{MR3034318}, one can see that  in the definition of $\I$-point one can consider only bijections instead of one-to-one functions functions i.e.~$\I\not\leq_{\point} \cU^* \iff \I\not\sqsubseteq \cU^*$.
\end{enumerate}

It was already mentioned in the Introduction that if there is an $\I$-ultrafilter which is not a $\J$-ultrafilter, then necessary $\I\not\leq_K\J$. Similarly, for instance, if there is an $\I$-ultrafilter (weak $\I$-ultrafilter or $\I$-point, resp.) which is not a $\J$-point, then $\I\not\leq_K\J$ ($\I\not\leq_{KB}\J$ or $\I\not\leq_{P}\J$, resp.). 

\begin{theorem}[{see e.g.~\cite[Observation~2.1]{MR3600759}}] \ 
\label{thm:known-ultrafilters-characterized-as-I-ultrafilters}
	\begin{enumerate}
		\item $\cU$ is a P-point 
		$\iff$ $\cU$ is a $\fin^2$-ultrafilter
		$\iff$ $\cU$ is a weak $\fin^2$-ultrafilter
		$\iff$ $\cU$ is a $\fin^2$-point
		$\iff$ $\cU$ is a $\conv$-ultrafilter
		$\iff$ $\cU$ is a weak $\conv$-ultrafilter
		$\iff$ $\cU$ is a $\conv$-point.
		\label{thm:known-ultrafilters-characterized-as-I-ultrafilters:P-point}

		\item $\cU$ is a Q-point 
$\iff$ $\cU$ is a weak $\ED_{\fin}$-ultrafilter
$\iff$ $\cU$ is a $\ED_{\fin}$-point.\label{thm:known-ultrafilters-characterized-as-I-ultrafilters:Q-point}

\item $\cU$ is a selective ultrafilter 
$\iff$ $\cU$ is an $\ED$-ultrafilter
$\iff$ $\cU$ is a weak $\ED$-ultrafilter
$\iff$ $\cU$ is an $\ED$-point
$\iff$ $\cU$ is a $\mathcal{R}$-ultrafilter
$\iff$ $\cU$ is a weak $\mathcal{R}$-ultrafilter
$\iff$ $\cU$ is a $\mathcal{R}$-point.\label{thm:known-ultrafilters-characterized-as-I-ultrafilters:Ramsey-point}

	\end{enumerate}
\end{theorem}


\subsection{Homogeneous ideals}

An ideal $\I$ is called \emph{homogeneous} if $\I\restriction A \approx \I$ for each $A\in \I^+$.
Similarly, an ideal $\I$ is called \emph{$\leq_{\cF}$-homogeneous} if $\I\restriction A \approx_{\cF} \I$ (i.e.~$\I\restriction A \leq_{\cF} \I$ and $\I \leq_{\cF} \I\restriction A$) for each $A\in \I^+$.
Homogeneous ideals were defined and examined in \cite{MR3594409}, whereas $\leq_{\Kat}$-homogeneous ideals were defined and examined in \cite{MR2777744,alcantara-phd-thesis} where the authors used the name \emph{$\Kat$-uniform} ideals in this case.

Obviously, homogeneity implies $\leq_{\point}$-homogeneity which implies $\leq_{\KB}$-homogeneity, and the latter implies 
$\leq_{\Kat}$-homogeneity in turn.
Moreover, it is easy to see that if $\cF$ contains the identity function, then $\I$ is $\leq_{\cF}$-homogeneous if and only if
$\I\restriction A \leq_{\cF} \I$ for each $A\in \I^+$.

\begin{proposition}\ 
\label{prop:homogeneous-idelas}
The ideals $\ED_{\fin}$, $\fin^2$  and $\cW$ are homogeneous.
\end{proposition}

\begin{proof}
See \cite[Example 2.4]{MR3594409}, \cite[Remark below Proposition 2.9]{MR3594409} and \cite[Example 2.6]{MR3594409}, respectively.
\end{proof}


\subsection{$P^+(\I)$-ideal}

Let $\I$ be an ideal on $X$.
We write $B\subseteq^{\I}A$  when $B\setminus A\in \I$ (if $\I=\fin$, we write $B\subseteq^* A$).
A set $B\in \I^+$ is an \emph{$\I^+$-pseudointersection} of a family $\cA$ if $B\subseteq^\I A$  for each $A\in \cA$ (if $\I=\fin$, we just say \emph{pseudointersection}).
An ideal $\I$ is a \emph{$P^+(\I)$-ideal} if for each decreasing sequence $A_0\supseteq A_1\supseteq \ldots$ of sets from $\I^+$ there is an $\I^+$-pseudointersection of the family $\{A_n:n\in \omega\}$.
Obviously, if $\I$ is a $P^+$-ideal, then it is $P^+(\I)$-ideal as well.	

\begin{proposition}\ 
	\label{prop:examples-of-P-plus-ideals}
	\begin{enumerate}
		\item Every $F_\sigma$ ideal is a $P^+$-ideal.\label{prop:examples-of-P-plus-ideals:F-sigma}
		\item $\fin^2$ is a $P^+(\fin^2)$-ideal, but it is not a $P^+$-ideal.\label{prop:examples-of-P-plus-ideals:FINxFIN} 		
	\end{enumerate}
\end{proposition}

\begin{proof}
	(\ref{prop:examples-of-P-plus-ideals:F-sigma}) See \cite[Lemma 3.2.4]{alcantara-phd-thesis}.
	
	(\ref{prop:examples-of-P-plus-ideals:FINxFIN})
	$\fin^2$ is not a $P^+$-ideal, because the sequence $A_n=(\omega\setminus n)\times\omega$ does not have a $(\fin)^+$-pseudointersection.
	Below, we show that $\fin^2$ is a $P^+(\fin^2)$-ideal.
	
	Let $A_n\in (\fin^2)^+$ be a decreasing sequence. 
	Then we can pick $k_0<k_1<\dots$ such that $B_n = (A_n)_{(k_n)}$ is infinite for each $n$.
	Then $$B=\bigcup_{n\in\omega}\{k_n\}\times B_n\in (\fin^2)^+$$
	and $B \setminus A_n \in \fin^2$ for each $n$. 
\end{proof}


\subsection{Pseudointersection numbers}

	A \emph{nonempty} family $\cA \subseteq\cP(\omega)$ has \emph{$\I^+$-FIP} if $\bigcap F\in \I^+$ for each finite $F\subseteq \cA$ (if $\I=\fin$, we say \emph{SFIP} instead of $\fin^+$-FIP).
	The \emph{pseudointersection number} of an ideal $\I$ on $\omega$ is defined by:
\begin{equation*}
	\begin{split}
\pnumber(\I) =  \min (\{|\cA|: 
\text{ $\cA$ has $\I^+$-FIP,} & \text{ but $\cA$ does not have }
\\& \text{an $\I^+$-pseudointersection}\} \cup\{\continuum^+\}).
\end{split}
\end{equation*}
The following easy proposition summarizes a few basic properties of $\pnumber(\I)$. 
\begin{proposition}\ 
\label{prop:pnumber}
\begin{enumerate}
	\item 
$\pnumber(\fin) = \pnumber$.
\item $\pnumber(\I) = \continuum^+$ for each maximal ideal $\I$.\label{prop:pnumber:maximal}

\item 
$\pnumber(\I)\geq \omega_1$ $\iff$ 
$\I$ is a $P^+(\I)$-ideal.\label{prop:pnumber:for-P-plus-ideals}
\end{enumerate}
\end{proposition}

\begin{theorem}[{\cite{MR698462}, see also \cite[Proposition~8.19]{MR2777744} or \cite[p.~897]{MR3395353}}]
\label{thm:pnumber-for-FINxFIN}
	$\pnumber(\fin^2) = \omega_1$.
\end{theorem}

The following theorem seems to be known (see~\cite{Brendle-slides} or \cite{Flaskova-poster}), however, we were unable to find a written proof of it anywhere, so we decided to include it for completeness.
\begin{theorem}
\label{thm:p-number-for-F-sigma}
	$\pnumber\leq \pnumber(\I)$ for each $F_\sigma$ ideal $\I$.
\end{theorem}

\begin{proof} 
 Let $\phi$ be a lsc submeasure such that $\I=\fin(\phi)$ (see Theorem~\ref{lscsm}). 
 Fix $\kappa<\pnumber$ and any $\cA\in  [\I^+]^{\kappa}$ having $\I^+$-FIP. We will find a set $B\in \I^+$ such that $B\setminus A\in \I$ for each $A\in \cA$. 
 
 We define a poset $(\poset,\leq)$ as follows:
$\poset=[\omega]^{<\omega} \times [\cA]^{<\omega}$ and $(s,F)\leq (t,G)\iff$
$s\supseteq t$, 
$F\supseteq G$ 
and
$s\setminus t\subseteq \bigcap G$. 
Now we define dense subsets of the poset:
\begin{enumerate}
\item $D_A=\{(s,F)\in \poset: A\in F\}$ for $A\in \cA$,
\item $E_n=\{(s,F)\in \poset: \phi(s)>n\}$ for every $n\in \omega$.
\end{enumerate}
Let $\cD = \{D_A: A\in\cA\} \cup \{E_n:n\in \omega\}$.
Since 
$\poset$ is $\sigma$-centered and  
$|\cD| <\pnumber$, using Bell's Theorem (\cite{MR643555}, see also~\cite[Theorem~7.12]{MR2768685} or \cite[Theorem~1.4.22]{MR1350295}), there exists a filter $\cG\subseteq \poset$ such that $\cG\cap D\neq \emptyset$ for each $D\in \cD$.

We define $B=\bigcup\{s\in [\omega]^{<\omega}: \text{$(s,F)\in \cG$ for some finite $F\subseteq  \cA$}\}$. Then $B\in \I^+$, because for each $n$ there is $(s,F)\in \cG\cap E_n$, so $\phi(B)\geq \phi(s)> n$.

Finally, $B\setminus A$ is finite for each $A\in \cA$. Indeed, if $A\in \cA$, then there is $(s,F)\in \cG\cap D_A$ and we will show that $B\setminus A\subseteq s$.

For $k\in B\setminus A$, we find $(s',F')\in \cG$ such that $k\in s'$. Since $\cG$ is a filter, there exists $(s'',F'')\in \cG$ such that $(s'',F'')\leq (s,F)$ and $(s'',F'')\leq (s',F')$.
Then $s''\setminus s \subseteq \bigcap F \subseteq A$ and $k\notin A$, so $k\notin s''\setminus s$.
Taking into account that $k\in s''$,  we obtain $k\in s$.
\end{proof}


\section{The cardinal characteristics and the order associated with $\I$-ultrafilters}
\label{sec:cardinal-nad-order}

In this section we develop tools which will be used in our further considerations -- we introduce two cardinal invariants associated to pairs of ideals and families of functions and show how they can be used in our studies.


\subsection{The cardinal $\wsp$}

\begin{definition}
	For ideals $\I,\J\subseteq\cP(\omega)$ and a family $\cF\subseteq\omega^\omega$ we define
	\begin{equation*}
		\begin{split}
			\wsp(\I,\J,\cF) = & \min (\{|\cA| :  \text{$\cA$ has $\J^+$-FIP and }
			\\& 
			\neg\left(\forall f\in \cF\,\exists B\in \I\,(\cA \cup \{f^{-1}[B]\} \text{ has $\J^+$-FIP})\right)\} \cup\{\continuum^+\}).
		\end{split}
	\end{equation*}
\end{definition}

The following theorem provides us a basic tool for constructing $\I$-ultrafilters which are not   $\J$-ultrafilters. 

\begin{theorem}
	\label{thm:wsp-implies-existence-of-I-ultrafilters}
	If $\wsp(\I,\J,\cF)\geq |\cF|$, then there exists an ultrafilter $\cU$ such that $\I\not\leq_{\cF}\cU^*$ and $\J\subseteq \cU^*$ . 
	
	In particular, 	if $\cF=\omega^\omega$ ($\cF$ is the family of all finite-to-one functions or $\cF$ is the family of all one-to-one functions, resp.), we obtain an $\I$-ultrafilter (weak $\I$-ultrafilter or $\I$-point, resp.) which is not a $\J$-point.
\end{theorem}

\begin{proof}
	Let $\cF = \{f_\alpha:\alpha<|\cF|\}$.
	We will construct a sequence $\{ A_\alpha:\alpha<|\cF|\} \subseteq\cP(\omega)$ such that for each $\alpha$ we have:
	\begin{enumerate}
		\item $\{A_\beta:\beta<\alpha\}$ has $\J^+$-FIP,
		\item $f_\alpha[A_\alpha]\in\I$.
	\end{enumerate}
	
	Suppose $A_\beta$ has been constructed for $\beta<\alpha$.
	We have 2 cases.
	
	Case 1. If there is $F\in [\alpha]^{<\omega}$ with $f_\alpha[\bigcap\{A_\beta:\beta\in F\}]\in \I$, we put $A_\alpha=\bigcap\{A_\beta:\beta\in F\}$.
	
	Case 2. Suppose that $f_\alpha[\bigcap\{A_\beta:\beta\in F\}]\notin \I$ for each $F\in [\alpha]^{<\omega}$.
	Since $|\alpha|<|\cF|\leq \wsp(\I,\J,\cF)$ and $\cA=\{A_\beta:\beta<\alpha\}$ has $\J^+$-FIP, there is $B\in \I$ such that $\cA\cup \{f_\alpha^{-1}[B]\}$ has $\J^+$-FIP.
	Then we put $A_\alpha = f_\alpha^{-1}[B]$.
	
	The construction of $\langle A_\alpha:\alpha<|\cF|\rangle$  is finished.
	
	Since $\{A_\alpha:\alpha<|\cF|\}$ has $\J^+$-FIP, the family  
	$\J^*\cup \{A_\alpha:\alpha<|\cF|\}$ has $\J^+$-FIP as well (in particular it has SFIP).
	Thus, there is an ultrafilter $\cU$ such that $\J^*\cup \{A_\alpha:\alpha<|\cF|\} \subseteq \cU$.
	It is easy to see that $\cU$ is the required ultrafilter.
\end{proof}


\subsection{The new order and the cardinal $\uu$}

For an ideal  $\I$ on $\omega$
and $\cA\subseteq\cP(\omega)$ we write
$$\I(\cA)=\left\{B\subseteq\omega:\exists F\in[\cA]^{<\omega} (B\setminus\bigcup F\in\I)\right\}.$$
Notice that either $\I(\cA) = \cP(\omega)$ or $\I(\cA)$  is the ideal generated by $\I\cup \cA$. The latter case holds exactly when $\omega\notin\I(\cA)$ (or, equivalently, when the family $\{\omega\setminus A:A\in \cA\}$ has $\I^+$-FIP). 

\begin{definition}
Let $\I$ and $\J$ be ideals on $\omega$, $\cF,\cG\subseteq\omega^\omega$  and $\lambda\leq\cc$.
\begin{enumerate}
	\item 
We write $\I\preceq^\lambda_{\cF,\cG}\J$ if for each ideal $\J'$ with $\J\leq_{\cG}\J'$ there exists $\cA \in [\cP(\omega)]^{<\lambda}$
such that $\omega\notin\J'(\cA)$ and $\I\leq_\cF\J'(\cA)$.
\item 	We define $$\uu(\I,\cF,\J,\cG)=\min\left(\left\{\lambda\leq\cc:\ \I\preceq^\lambda_{\cF,\cG} \J\right\}\cup\left\{\cc^+\right\}\right).$$
\end{enumerate}
\end{definition}

The following theorem provides us a basic tool for proving when the existence of $\I$-ultrafilters which are not   $\J$-ultrafilters is equivalent to $\I\nleq_{\Kat}\J$.

\begin{theorem}
\label{thm:unumber-iff-I-ultra-not-J-ultra}	
	If $\cG$ contains the identity function, then 
\begin{equation*}
	\begin{split}
		\uu(\I,\cF,\J,\cG)> |\cF| \iff & \text{ there exists an ultrafilter $\cU$ such that} 
		\\& \text{ $\I\not\leq_{\cF}\cU^*$ and $\J\leq_{\cG}\cU^*$.}
	\end{split}
\end{equation*}
\end{theorem}

\begin{proof}
$(\Leftarrow)$ 
If $\cU$ is an ultrafilter 
such that $\I\not\leq_{\cF}\cU^*$ and $\J\leq_{\cG}\cU^*$, then 
$\J' = \cU^*$ is a maximal ideal $\leq_{\cG}$-above $\J$ which witnesses $\I\not\preceq^{\lambda}_{\cF,\cG}\J$ for every $\lambda\leq \continuum$. Hence $\uu(\I,\cF,\J,\cG)=\continuum^+ > |\cF|$.

$(\Rightarrow)$
If $\uu(\I,\cF,\J,\cG) >  |\cF|$, then 
we have $\I\not\preceq^{|\cF|}_{\cF,\cG}\J$. Thus, there is an ideal $\J'$ such that $\J\leq_{\cG}\J'$
and
$\I\not\leq_\cF\J'(\cA)$
for every $\cA \in [\cP(\omega)]^{< |\cF|}$
with $\omega\notin\J'(\cA)$ .

	Let $\cF = \{f_\alpha:\alpha<|\cF|\}$.
We will construct a sequence $\langle A_\alpha:\alpha<|\cF|\rangle$ of subsets of $\omega$ such that for each $\alpha$ we have:
\begin{enumerate}
	\item $\{\omega\setminus A_\beta:\beta<\alpha\}$ has $(\J')^+$-FIP,
	\item $f_\alpha[\omega\setminus A_\alpha]\in\I$.
\end{enumerate}

Suppose $A_\beta$ has been constructed for $\beta<\alpha$.
Let $\cA=\{A_\beta:\beta<\alpha\}$.
Since  $\{\omega\setminus A_\beta:\beta<\alpha\}$  has $(\J')^+$-FIP, 
$\omega\notin\J'(\cA)$.
Since $|\cA|\leq |\alpha|<|\cF|$, 
$\I\not\leq_\cF\J'(\cA)$.
Thus, there is $B\in (\J'(\cA))^+$ such that $f_\alpha[B]\in\I$.
Then we put $A_\alpha = \omega\setminus  B$, and the construction of $\langle A_\alpha:\alpha<|\cF|\rangle$  is finished.

Since $\{\omega\setminus  A_\alpha:\alpha<|\cF|\}$ has $(\J')^+$-FIP, the family  
$(\J')^*\cup \{\omega\setminus A_\alpha:\alpha<|\cF|\}$ has $(\J')^+$-FIP as well (in particular it has SFIP).
Thus, there is an ultrafilter $\cU$ such that $(\J')^*\cup \{\omega\setminus A_\alpha:\alpha<|\cF|\} \subseteq \cU$.
It is easy to see that $\cU$ is the required ultrafilter.
\end{proof}

\begin{corollary}\label{cor:new-order-iff-existence-of-ultrafilters}
Let $\I$ and $\J$ be ideals on $\omega$.
\begin{enumerate}
\item $\I\not\preceq_{\omega^\omega,\omega^\omega}^{\continuum}\J$
	($\I\not\preceq_{\omega^\omega,\fin-1}^{\continuum}\J$ or $\I\not\preceq^{\continuum}_{\omega^\omega,1-1}\J$, resp.)
	 $\iff$ there exists  an $\I$-ultrafilter which is not a $\J$-ultrafilter (weak $\J$-ultrafilter or $\J$-point, resp.).\label{cor:new-order-iff-existence-of-ultrafilters:Katetov}

	\item $\I\not\preceq_{\fin-1,\fin-1}^{\continuum}\J$ ($\I\not\preceq_{\fin-1,1-1}^{\continuum}\J$, resp.)
	 $\iff$ 
	 there exists a weak $\I$-ultrafilter which is not a weak $\J$-ultrafilter ($\J$-point, resp.).

	\item $\I\not\preceq_{1-1,1-1}^{\continuum}\J$
 $\iff$ there exists an $\I$-point which is not a $\J$-point.

\end{enumerate}
\end{corollary}

\begin{remark}	It is easy to see that $\I\leq_{\Kat}\J$ implies $\I\preceq_{\omega^\omega,\omega^\omega}^{\continuum}\J$. However, the inverse implication is not true in general. Indeed,  it is known that $\fin^2\not\leq_{\Kat}\conv$ \cite{MR3696069}, but $\fin^2\preceq_{\omega^\omega,\omega^\omega}^{\continuum}\conv$ holds as $\fin^2$-ultrafilters and  $\conv$-ultrafilters and P-points all are the same notion (see Theorem~\ref{thm:known-ultrafilters-characterized-as-I-ultrafilters}(\ref{thm:known-ultrafilters-characterized-as-I-ultrafilters:P-point})). 
\end{remark}

We end this section with a result connecting the two cardinal invariants introduced above.

\begin{theorem}
	\label{the:u-greater-than-wsp}
	If $\cG$ contains the identity function, then 
	either 
	$\wsp(\I,\J,\cF) = \uu(\I,\cF,\J,\cG) = \continuum^+$
	or
	$$\wsp(\I,\J,\cF) < \uu(\I,\cF,\J,\cG).$$
\end{theorem}

\begin{proof}
Assume first that $\wsp(\I,\J,\cF) < \continuum^+$.
	We will show that $\I\not\preceq_{\cF,\cG}^{\wsp(\I,\J,\cF)} \J$.
	Let $\J'=\J$ and take any $\cA\subseteq\cP(\omega)$ such that $|\cA|<\wsp(\I,\J,\cF)$ and $\omega\notin\J'(\cA)$.
	Let $\cB=\{\omega\setminus A:A\in\cA\}$.
	Then $\cB$ has $\J^+$-FIP and $|\cB|<\wsp(\I,\J,\cF)$, 
	so for each $f\in \cF$ there is $C\in \I$ such that $\cB\cup\{f^{-1}[C]\}$ has $\J^+$-FIP. In particular, $f^{-1}[C]\notin\J'(\cA)$. Consequently, $\I\not\leq_{\cF}\J'(\cA)$.
	
	If $\wsp(\I,\J,\cF) =  \continuum^+$, then it is not difficult to see that 
	$\I\npreceq^\lambda_{\cF,\cG}\J$ for each $\lambda\leq \continuum$, so $\uu(\I,\cF,\J,\cG) = \continuum^+$.
\end{proof}


\section{More on the cardinals $\wsp$ and $\uu$}
\label{sec:lower-bound}

\subsection{Some lower bounds}

\begin{theorem}
\label{thm:wsp-vs-p-number-and-addSTAR}
Let $\cF\subseteq\omega^\omega$ and $\I$, $\J$ be ideals on $\omega$
such that $\I$ is tall and $\I\not\leq_{\cF} \J\restriction A$ for all $A\notin\J$. We have:
\begin{enumerate}
	\item 
$\pnumber(\J) \leq \wsp(\I,\J,\cF)$.	\label{thm:wsp-vs-p-number-and-addSTAR:pnumber}
	\item 
$\adds(\I) \leq \wsp(\I,\J,\cF)$.	\label{thm:wsp-vs-p-number-and-addSTAR:addSTAR}
\end{enumerate}
\end{theorem}

\begin{proof}
(\ref{thm:wsp-vs-p-number-and-addSTAR:pnumber})
	Let $\cA$ be a family having $\J^+$-FIP and such that  $|\cA|<\pnumber(\J)$.
	Then there is $C\notin\J$ such that $C\subseteq^{\J}A$ for each $A\in \cA$.
	In particular, 
	$C\subseteq^{\J} \bigcap F$ for each $F\in [\cA]^{<\omega}$.
	Let $f\in \cF$.
	Since $\I\not\leq_{\cF}\J\restriction C$, there is $D\subseteq C$, $D\notin \J$ with $f[D]\in \I$.
	We put $B = f[D]$ and notice that for each $F\in [\cA]^{<\omega}$,
	$\bigcap F \cap f^{-1}[B] \supseteq \bigcap \cF \cap D\notin\J$. It means that $\cA\cup\{f^{-1}[B]\}$,	has $\J^+$-FIP, so $|\cA| < \wsp(\I,\J,\cF)$.

(\ref{thm:wsp-vs-p-number-and-addSTAR:addSTAR})
Let $\cA$ be a family having $\J^+$-FIP and such that  $|\cA|<\adds(\I)$.
Fix any $f\in \cF$.

If there is $n\in \omega$ such that $\{f^{-1}[\{n\}]\}\cup\cA$ has $\J^+$-FIP, then 
$|\cA| < \wsp(\I,\J,\cF)$ and the proof is finished.
Otherwise,
$\{f^{-1}[\{n\}]\}\cup\cA$ does not have $\J^+$-FIP for any $n\in \omega$.
Then, for each $F\in [\cA]^{<\omega}$, we define
$$Y_1^F = \{n\in \omega: f^{-1}[\{n\}]\cap \bigcap F\in\J\} \text{\ \ \ and \ \ \ } Y_2^F = \omega\setminus Y_1^F,$$
and pick a set $B_F\in\I$ in the following manner:
\begin{itemize}
\item If $\{f^{-1}[Y_1^F]\}\cup\cA$ has $\J^+$-FIP, then $A = f^{-1}[Y_1^F]\cap \bigcap F\notin\J$.
Since $\I\not\leq_{\cF}\J\restriction A$, there is $B_F\in \I$ such that $B_F\subseteq f[f^{-1}[Y_1^F]\cap \bigcap F]$ and $f^{-1}[B_F]\cap (f^{-1}[Y_1^F]\cap \bigcap F)\notin\J$.

\item 
If $\{f^{-1}[Y_2^F]\}\cup\cA$ has $\J^+$-FIP, then $Y_2^F$ is infinite (because $\{f^{-1}[\{n\}]\}\cup\cA$ does not have $\J^+$-FIP for each $n$), so there is infinite $B_F\subseteq Y_2^F$ such that $B_F\in \I$ (because $\I$ is tall).
\end{itemize}

Since $|[\cA]^{<\omega}|<\adds(\I)$, there is $B\in \I$ such that $B_F\setminus B$ is finite for each 
$F\in [\cA]^{<\omega}$.
We claim that $\{f^{-1}[B]\}\cup\cA$ has $\J^+$-FIP (if so, then 
$|\cA| < \wsp(\I,\J,\cF)$ and the proof will be finished).

Take any $F\in [\cA]^{<\omega}$.
We have two cases.

\emph{Case 1.}
The family $\{f^{-1}[Y_1^F]\}\cup\cA$ has $\J^+$-FIP.

Then $f^{-1}[B_F]\cap \bigcap F\notin\J$.
Moreover, $B_F\setminus B$ is finite and $f^{-1}[\{n\}]\cap \bigcap F\in \J$ for each $n\in B_F\setminus B$ (as $B_F\subseteq f[f^{-1}[Y_1^F]\subseteq Y_1^F$), so $f^{-1}[B_F\setminus B]\cap \bigcap F \in \J$.
Thus, $f^{-1}[B]\cap \bigcap F \supseteq f^{-1}[B_F\cap B]\cap \bigcap F\notin\J$.

\emph{Case 2.}
The family $\{f^{-1}[Y_2^F]\}\cup\cA$ has $\J^+$-FIP.

Then $B_F$ is infinite, so $B_F\cap B\neq\emptyset$ as $B_F\setminus B$ is finite.
Take any $n\in B_F\cap B$.
Then $f^{-1}[\{n\}]\cap \bigcap F\notin\J$ (as $n\in B_F\subseteq Y_2^F$),
and $f^{-1}[B]\supseteq f^{-1}[\{n\}]$ (as $n\in B$).
Thus, 
$f^{-1}[B]\cap \bigcap F \notin\J$.
\end{proof}

As a simple corollary, we obtain a result of Fla\v{s}kov\'{a} from \cite{MR2512901}.

\begin{corollary}[{\cite[Proposition~2.3]{MR2512901}}]
\label{cor:p=c-implies-exists-I-ultrafilter-for-tall-I}
	Assume $\pnumber=\continuum$. There exists an $\I$-ultrafilter for each tall ideal $\I$.	
\end{corollary}

\begin{proof}
	Since $\pnumber(\fin)=\pnumber$ and $\I\not\leq_{\Kat} \fin \approx \fin\restriction A$ for each $A\notin \fin$, it is enough to apply Theorems~\ref{thm:wsp-implies-existence-of-I-ultrafilters} and \ref{thm:wsp-vs-p-number-and-addSTAR}(\ref{thm:wsp-vs-p-number-and-addSTAR:pnumber}).	
\end{proof}


\subsection{Connection with the Kat\v{e}tov order}

The following result, which was already mentioned in the Introduction, provides sufficient conditions on ideals to obtain the equivalence:
$\I\not\leq_K\J$ if and only if it is consistent that there exists an $\I$-ultrafilter which
is not a $\J$-ultrafilter.

\begin{theorem}
	\label{thm:KAT-iff-EXISTS-ULTRA-for-K-uniform-Pplus-ideals}
	Let $\I$ be a tall ideal and $\J$ be a $\text{P}^+(\J)$-ideal.  
		\begin{enumerate}
		\item 	If $\J$ is $\leq_{\Kat}$-homogeneous, then under CH, the following conditions are equivalent. \label{thm:KAT-iff-EXISTS-ULTRA-for-K-uniform-Pplus-ideals:CH-KAT}
	\begin{enumerate}
		\item $\I\not\leq_{\Kat}\J$.
		\item $\I\not\preceq_{\omega^\omega,\omega^\omega}^{\continuum}\J$.
 \item $\I\not\preceq_{\omega^\omega,\fin-1}^{\continuum}\J$.
		\item $\I\not\preceq_{\omega^\omega,1-1}^{\continuum}\J$. 		
		\item There exists an $\I$-ultrafilter  which is not a $\J$-point.\label{thm:KAT-iff-EXISTS-ULTRA-for-K-uniform-Pplus-ideals:CH-KAT-Ultra}
		\item There exists an $\I$-ultrafilter which is not a weak $\J$-ultrafilter.\label{thm:KAT-iff-EXISTS-ULTRA-for-K-uniform-Pplus-ideals:CH-KAT-Ultra1}
		\item There exists an $\I$-ultrafilter which is not a $\J$-ultrafilter.\label{thm:KAT-iff-EXISTS-ULTRA-for-K-uniform-Pplus-ideals:CH-KAT-Ultra2}
	\end{enumerate}

\item If $\J$ is $\leq_{\KB}$-homogeneous, then under CH, the following conditions are equivalent.\label{thm:KAT-iff-EXISTS-ULTRA-for-K-uniform-Pplus-ideals:CH-KB}
\begin{enumerate}
	\item $\I\not\leq_{\KB}\J$.
	\item $\I\not\preceq_{\fin-1,\fin-1}^{\continuum}\J$.
	\item $\I\not\preceq_{\fin-1,1-1}^{\continuum}\J$.
	\item There exists a weak $\I$-ultrafilter which is not a $\J$-point. \label{thm:KAT-iff-EXISTS-ULTRA-for-K-uniform-Pplus-ideals:CH-KB-Ultra}
	\item There exists a weak $\I$-ultrafilter which is not a weak $\J$-ultrafilter.\label{thm:KAT-iff-EXISTS-ULTRA-for-K-uniform-Pplus-ideals:CH-KB-Ultra1}
\end{enumerate}

\item 	
If $\J$ is $\leq_{\point}$-homogeneous, then under CH, the following conditions are equivalent.\label{thm:KAT-iff-EXISTS-ULTRA-for-K-uniform-Pplus-ideals:CH-P}
\begin{enumerate}
	\item $\I\not\leq_{\point}\J$.
	\item $\I\not\preceq_{1-1,1-1}^{\continuum}\J$.
	\item There exists an $\I$-point which is not a $\J$-point.\label{thm:KAT-iff-EXISTS-ULTRA-for-K-uniform-Pplus-ideals:CH-P-Utra}
\end{enumerate}

\item 	If $\J$ is an $F_\sigma$ ideal, then CH can be relaxed to the assumption $\pnumber=\continuum$ in (\ref{thm:KAT-iff-EXISTS-ULTRA-for-K-uniform-Pplus-ideals:CH-KAT}), (\ref{thm:KAT-iff-EXISTS-ULTRA-for-K-uniform-Pplus-ideals:CH-KB}) and (\ref{thm:KAT-iff-EXISTS-ULTRA-for-K-uniform-Pplus-ideals:CH-P}).\label{thm:KAT-iff-EXISTS-ULTRA-for-K-uniform-Pplus-ideals:Borel}

\item If $\I$ and $\J$ are Borel and $\J$ is $\leq_{\Kat}$-homogeneous ($\leq_{\KB}$-homogeneous or $\leq_{\point}$-homogeneous, resp.), then the following conditions are equivalent.\label{thm:KAT-iff-EXISTS-ULTRA-for-K-uniform-Pplus-ideals:CON}
	\begin{enumerate}
\item It is consistent that there exists an $\I$-ultrafilter (weak $\I$-ultrafilter or $\I$-point, resp.) which is not a $\J$-point.\label{thm:KAT-iff-EXISTS-ULTRA-for-K-uniform-Pplus-ideals:CON-CON}

\item Under CH, there exists an $\I$-ultrafilter (weak $\I$-ultrafilter or $\I$-point, resp.) which is not a $\J$-point.\label{thm:KAT-iff-EXISTS-ULTRA-for-K-uniform-Pplus-ideals:CON-CH}
\end{enumerate}
In particular, if one can find an $\I$-ultrafilter (weak $\I$-ultrafilter or $\I$-point, resp.) which is not a $\J$-point in some sophisticated model of ZFC, then one could do it already under CH.

\item Equivalences similar to item (\ref{thm:KAT-iff-EXISTS-ULTRA-for-K-uniform-Pplus-ideals:CON}) (which concerns items (\ref{thm:KAT-iff-EXISTS-ULTRA-for-K-uniform-Pplus-ideals:CH-KAT-Ultra}), (\ref{thm:KAT-iff-EXISTS-ULTRA-for-K-uniform-Pplus-ideals:CH-KB-Ultra}), and (\ref{thm:KAT-iff-EXISTS-ULTRA-for-K-uniform-Pplus-ideals:CH-P-Utra}) above)
hold for counterparts of items (\ref{thm:KAT-iff-EXISTS-ULTRA-for-K-uniform-Pplus-ideals:CH-KAT-Ultra1}), (\ref{thm:KAT-iff-EXISTS-ULTRA-for-K-uniform-Pplus-ideals:CH-KAT-Ultra2}) and (\ref{thm:KAT-iff-EXISTS-ULTRA-for-K-uniform-Pplus-ideals:CH-KB-Ultra1}).\label{thm:KAT-iff-EXISTS-ULTRA-for-K-uniform-Pplus-ideals:CON-other-orders}
\end{enumerate}
\end{theorem}

\begin{proof}
(\ref{thm:KAT-iff-EXISTS-ULTRA-for-K-uniform-Pplus-ideals:CH-KAT}) Taking into account the obvious implications, 
	using Corollary~\ref{cor:new-order-iff-existence-of-ultrafilters} and knowing that $\I\not\preceq_{\omega^\omega,\omega^\omega}^{\continuum}\J$ implies $\I\not\leq_{\Kat}\J$,  the proof will be finished once we show that $\I\not\leq_{\Kat}\J$  implies that there is an $\I$-ultrafilter which is not a $\J$-point.
	
	Assume that $\I\not\leq_{\Kat}\J$. Since $\J$ is $\leq_{\Kat}$-homogeneous, we obtain $\I\not\leq_{\Kat} \J\restriction A$ for all $A\notin\J$. 
	Now, Theorem~\ref{thm:wsp-vs-p-number-and-addSTAR}  implies $\wsp(\I,\J,\omega^\omega)\geq \pnumber(\J)$.
	Since $\J$ is a $P^+(\J)$-ideal, $\pnumber(\J)\geq \omega_1$ by Proposition~\ref{prop:pnumber}(\ref{prop:pnumber:for-P-plus-ideals}).
Thus, using CH and applying Theorem~\ref{thm:wsp-implies-existence-of-I-ultrafilters}, we obtain an $\I$-ultrafilter which is not a $\J$-point.

The proofs of items (\ref{thm:KAT-iff-EXISTS-ULTRA-for-K-uniform-Pplus-ideals:CH-KB}) and (\ref{thm:KAT-iff-EXISTS-ULTRA-for-K-uniform-Pplus-ideals:CH-P}) are very similar to the proof of item (\ref{thm:KAT-iff-EXISTS-ULTRA-for-K-uniform-Pplus-ideals:CH-KAT}) as one can easily check that $\I\not\preceq_{\fin-1,\fin-1}^{\continuum}\J$ ($\I\not\preceq_{1-1,1-1}^{\continuum}\J$, resp.) implies $\I\not\leq_{KB}\J$ ($\I\not\leq_{P}\J$, resp.).
	
(\ref{thm:KAT-iff-EXISTS-ULTRA-for-K-uniform-Pplus-ideals:Borel}) It follows from the proofs of (\ref{thm:KAT-iff-EXISTS-ULTRA-for-K-uniform-Pplus-ideals:CH-KAT}), (\ref{thm:KAT-iff-EXISTS-ULTRA-for-K-uniform-Pplus-ideals:CH-KB}), (\ref{thm:KAT-iff-EXISTS-ULTRA-for-K-uniform-Pplus-ideals:CH-P}) and the fact that $\pp(\J)\geq\pp$ for $F_\sigma$ ideals (Theorem~\ref{thm:p-number-for-F-sigma}). 

(\ref{thm:KAT-iff-EXISTS-ULTRA-for-K-uniform-Pplus-ideals:CON})
The implication ``(\ref{thm:KAT-iff-EXISTS-ULTRA-for-K-uniform-Pplus-ideals:CON-CH})$\implies$(\ref{thm:KAT-iff-EXISTS-ULTRA-for-K-uniform-Pplus-ideals:CON-CON})'' is obvious, so below we show only the implication ``(\ref{thm:KAT-iff-EXISTS-ULTRA-for-K-uniform-Pplus-ideals:CON-CON})$\implies$(\ref{thm:KAT-iff-EXISTS-ULTRA-for-K-uniform-Pplus-ideals:CON-CH})''.
Assume that it is consistent that there exists an $\I$-ultrafilter (weak $\I$-ultrafilter or $\I$-point, resp.)  which is not a $\J$-point.
Then, obviously, it is consistent that $\I\not\leq_{\Kat}\J$ ($\I\not\leq_{\KB}\J$, $\I\not\leq_{\point}\J$, resp.) -- see Subsection \ref{subsec:ultrafilters}.
However, if $\I$ and $\J$ are Borel ideals, then the sentence $\I\nleq_{\Kat}\J$ ($\I\not\leq_{\KB}\J$, $\I\not\leq_{\point}\J$, resp.) is absolute (see \cite[p.~210]{MR3600759}), so $\I\not\leq_{\Kat}\J$ ($\I\not\leq_{\KB}\J$, $\I\not\leq_{\point}\J$, resp.) holds in ZFC.
Then, using item (\ref{thm:KAT-iff-EXISTS-ULTRA-for-K-uniform-Pplus-ideals:CH-KAT}) ((\ref{thm:KAT-iff-EXISTS-ULTRA-for-K-uniform-Pplus-ideals:CH-KB}) or (\ref{thm:KAT-iff-EXISTS-ULTRA-for-K-uniform-Pplus-ideals:CH-P}), resp.) we obtain, under CH, 
that there exists an $\I$-ultrafilter (weak $\I$-ultrafilter or $\I$-point, resp.) which is not a $\J$-point.

(\ref{thm:KAT-iff-EXISTS-ULTRA-for-K-uniform-Pplus-ideals:CON-other-orders})
This can be shown in the same way as item (\ref{thm:KAT-iff-EXISTS-ULTRA-for-K-uniform-Pplus-ideals:CON}).
\end{proof}


\subsection{Maximal ideals}

\begin{theorem}
	\label{thm:wsp-vs-character}
	Let $\I$ be a maximal ideal on $\omega$.
	If $\cF\subseteq\omega^\omega$ contains the identity function, then 
	$$\wsp(\I,\fin,\cF) = \chi(\I^*),$$ 
	where 
	$\chi(\I^*) = \min\{|\cB|:\cB\subseteq\I^*\land \forall A\in \I^*\,\exists B\in \cB\,(B\subseteq^* A)\}$.
\end{theorem}

\begin{proof}
	First we show $\wsp(\I,\fin,\cF) \geq  \chi(\I^*)$.
	Take a family $\cA$ with SFIP and such that $|\cA|< \chi(\I^*)$.
	Take any $f\in \cF$.
	We have 2 cases.
	
	Case 1. There is $F\in [\cA]^{<\omega}$ with $f[\bigcap F]\in \I$. Then we put $B = f[\bigcap F]$ and observe that
	$\bigcap G\cap f^{-1}[B] \supseteq \bigcap G\cap \bigcap F \notin \fin$ for each $G\in [\cA]^{<\omega}$.
	
	Case 2. For each $F\in [\cA]^{<\omega}$, $f[\bigcap F]\notin \I$. 
	Since $\I$ is maximal, 
	$f[\bigcap F]\in \I^*$ for each $F\in [\cA]^{<\omega}$. 
	Since $|[\cA]^{<\omega}|=|\cA| <\chi(\I^*)$, there is $C\in \I^*$ such that $f[\bigcap F]\not\subseteq^*C$  for each $F\in [\cA]^{<\omega}$. 
	Then we put $B=\omega\setminus C$.
	Since $\bigcap F\cap f^{-1}[B] = \bigcap F\setminus f^{-1}[C]$
	for each $F\in [\cA]^{<\omega}$, 
	$|\bigcap F\cap f^{-1}[B] | \geq |f[\bigcap F]\setminus C|=\omega$. 
	
	We conclude that $\wsp(\I,\fin,\cF)> |\cA|$ and consequently $\wsp(\I,\fin,\cF) \geq  \chi(\I^*)$.
	
	
 Now we show $\wsp(\I,\fin,\cF) \leq  \chi(\I^*)$.
	Fix $\cG\subseteq\I^*$ with $|\cG|<\wsp(\I,\fin,\cF)$.
	Since $\cG$ has SFIP and $|\cG|<\wsp(\I,\fin,\cF)$, for the identity function $f\in \cF$, we can find  $B\in \I$ such that $\cG\cup\{B\} = \cG\cup\{f^{-1}[B]\}$ has SFIP.
	Let $C=\omega\setminus B$.
	Then $C\in \I^*$ and for each $G\in \cG$ we have 
	$|G\setminus C| = |G\cap B| = \omega$, so $G\not\subseteq^* C$.
	Thus $\chi(\I^*)>|\cG|$, and consequently,  	$\wsp(\I,\fin,\cF) \leq  \chi(\I^*)$.
\end{proof}

Again as a simple corollary, we obtain another result of Fla\v{s}kov\'{a} from \cite{MR2512901}.

\begin{corollary}[{\cite[Proposition~2.1]{MR2512901}}]
	There exists an $\I$-ultrafilter
	for each maximal ideal $\I$ such that $\chi(\I^*)=\continuum$.
\end{corollary}

\begin{proof}
	Apply Theorems~\ref{thm:wsp-implies-existence-of-I-ultrafilters} and \ref{thm:wsp-vs-character}.	
\end{proof}


\section{Let's play a game}
\label{sec:game}

In this section we develop an infinite game which will enable us to replace CH or $\pnumber=\continuum$ with $\cov(\cM)=\continuum$ in some of our results. 

\begin{definition}
Consider the following game $G(\J,A,f)$ associated to an ideal $\J$ on $\omega$, $A\notin\J$ and a function $f:A\to\omega$ which is $\J$-to-one (i.e.~$f^{-1}[\{n\}]\in \J$ for every $n$). In the first move Player I plays $k_0\in\omega$ and then in the $n$th move Player II plays $G_n\in\fin$ and Player I responses in the $(n+1)$st move with a pair $(F_{n},k_{n+1})\in\fin\times\omega$ such that $F_n\cap G_n=\emptyset$ and $|F_n|\leq k_n$. At the end Player I is declared the winner if $f^{-1}[\bigcup_{n\in\omega}F_n]\notin\J\restriction A$. Otherwise Player II is the winner.
\end{definition}

Next result characterizes winning strategies in $G(\J,A,f)$ and connects the considered game with the property of containing a tall summable ideal.

\begin{lemma}\ 
	\label{lem:gra}
\begin{enumerate}
\item Player I has a winning strategy in $G(\J,A,f)$ if and only if 
there are families $\cF_0,\cF_1,\dots\subseteq\fin$ such that:\label{lem:gra:win-strategy-for-I}
	\begin{enumerate}
	\item $\sup \{|F|: F\in \cF_n \}< \infty$ for each $n$,
	\item $\sup \{\min(F): F\in\cF_n\}=\infty$ for each $n$,
	\item if $F_n\in\cF_n$ for each $n$, then $f^{-1}[\bigcup_{n\in\omega} F_n]\notin\J\restriction A$.
\end{enumerate}	
\item Player II has a winning strategy in $G(\J,A,f)$ if and only if there is a tall summable ideal $\I_g$ such that $f$ witnesses $\I_g\leq_{\Kat}\J\restriction A$.\label{lem:gra:win-strategy-for-II}
\item If $\J\restriction A$ is a Borel ideal, then the game $G(\J,A,f)$ is determined for all $f$.\label{lem:gra:determined-fixed-set}
\item If $\J$ is Borel such that for every $A\notin\J$ the ideal $\J\restriction A$ is not $\leq_{\Kat}$-above any tall summable ideal, then Player I has a winning strategy in $G(\J,A,f)$, for all $A$ and $f$.\label{lem:gra:determined-all-sets}
\end{enumerate}
\end{lemma}

\begin{proof}
(\ref{lem:gra:win-strategy-for-I}) Assume that there are families $\cF_n$ satisfying (a)-(c). Then Player I has a winning strategy by playing in his initial move $k_0 = \sup \{|F|: F\in \cF_0 \}$ and then in each next move (in response to some $G_n\in\fin$) any $F_n\in\cF_n$ such that $F_n\cap G_n=\emptyset$ (such $F_n$ exists from the property (b)) and $k_{n+1} = \sup \{|F|: F\in \cF_{n+1}\}$.

Now, assume that Player I has a winning strategy $\$:\fin^{<\omega}\to\fin\times \omega$.
Let $\pi_\fin:\fin\times\omega\to\fin$ and $\pi_\omega:\fin\times\omega\to\omega$ be the projections onto the first and the second coordinate, respectively (so $\pi_\fin(F,n)=F$ and $\pi_\omega(F,n)=n$ for all $(F,n)\in\fin\times\omega$). Let $\cF(\emptyset)=\{\pi_\fin(\$(G)): G\in\fin\}$ and 
$$\cF(G_0,\ldots,G_n)=\{\pi_\fin(\$(G_0,\ldots,G_n,G)):G\in\fin\}$$ 
for all $(G_0,\ldots,G_n)\in\fin^{<\omega}$. Observe that for each $F\in\cF(G_0,\ldots,G_n)$ we have $|F|\leq\pi_\omega(\$(G_0,\ldots,G_n))$, and $\sup \{\min(F):F\in\cF(G_0,\ldots,G_n)\}=\infty$ as for each $m\in\omega$ minimum of the set $\pi_\fin(\$(G_0,\ldots,G_n,[0,m]))\in\cF(G_0,\ldots,G_n)$ is greater than $m$. Thus, the families $\cF(G_0,\ldots,G_n)$ satisfy (a) and (b). As $\fin^{<\omega}$ is countable, it suffices to check (c). For each $(G_0,\ldots,G_n)\in\fin^{<\omega}$, take $F_{(G_0,\ldots,G_n)}\in\cF(G_0,\ldots,G_n)$. 
Let $B = \bigcup\{F_{(G_0,\ldots,G_n)}:(G_0,\ldots,G_n)\in\fin^{<\omega}\}$. 
To finish the proof, we have to show that $f^{-1}[B]\notin\J\restriction A$.

Since $F_{\emptyset}\in\cF(\emptyset)$, there is $H_0\in\fin$ such that $F_{\emptyset}=\pi_\fin(\$(H_0))$. 
Since $F_{(H_0)}\in\cF(H_0)$, there is $H_1\in\fin$ such that $F_{(H_0)}=\pi_\fin(\$(H_0,H_1))$. 
Continuing in this way we can inductively pick an infinite  sequence $H_0,H_1,\ldots \in\fin$ such that $F_{(H_0,\dots,H_n)}=\pi_\fin(\$(H_0,\ldots,H_n, H_{n+1}))\in \cF(H_0,\dots,H_n)$ for each $n\in\omega$. 
Then 
$$f^{-1}[B]\supseteq  f^{-1}\left[\bigcup_{n\in\omega}F_{(H_0,\dots,H_n)}\right]\notin\J\restriction A$$ as $\bigcup_{n\in\omega}F_{(H_0,\dots,H_n)}$ constitutes an outcome of a play of the game along the winning strategy for Player I.

(\ref{lem:gra:win-strategy-for-II}) 
Suppose that $f:A\to\omega$ witnesses $\I_g\leq_{\Kat}\J\restriction A$. Then Player II has a winning strategy by playing in each move a set $[0,m_n]\in\fin$ such that $k_{n}g(i)<1/2^n$ for all $i>m_n$ as in this case $$\sum_{i\in\bigcup_{n\in\omega}F_n}g(i) \leq  \sum_{n\in\omega}\frac{1}{2^n}<+\infty,$$
hence $\bigcup_{n\in\omega}F_n\in\I_g$ and consequently $f^{-1}[\bigcup_{n\in\omega}F_n] \in\J\restriction A$.

Now, assume that Player II has a winning strategy $\$:\omega\times (\fin\times\omega)^{<\omega}\to\fin$.

Put $M_0=\$(1)$, $M_1=\$(2)$ and define inductively:
\begin{equation*}
	\begin{split}
M_{n+1} =
&\bigcup
\{\$(k_0,F_0,k_1,F_1\ldots,k_l,F_l,k_{l+1}):l\leq n,\forall i\leq l\left(|F_i|\leq k_i\right),
\\&
\forall i\leq l\left(F_i\cap\$(k_0,F_0,\ldots,k_{i})=\emptyset\right),\forall i\leq l+1\left(k_i\leq n+2\right),\bigcup_{i\leq n}F_i\subseteq M_{n}\}
\end{split}
\end{equation*}
for all $n>0$. Without loss of generality we may assume that $(M_n)_{n\in \omega}$ is a strictly increasing sequence of finite sets and $\bigcup_{n\in\omega}M_n=\omega$ (as playing larger sets by Player II only increases  chances of winning). 

Define $g:\omega\to[0,+\infty)$ by $g(i) = \frac{1}{n+1}$ whenever $i\in M_{n+1}\setminus M_n$ and $g(i)=2$ for all $i\in M_0$. Observe that $\I_g$ is tall as $\lim_{i\to \infty} g(i)=0$. We claim that $f$ witnesses $\I_g\leq_{\Kat}\J\restriction A$. 

For every $B\in\I_g$ there is $k\in\omega$ such that 
$$\frac{|B\cap M_n|}{n}\leq\sum_{i\in B\cap M_n}g(i)\leq k,$$ 
for all $n>0$. Hence, $B$ can be partitioned into $M_0\cap B$ (for which we have $f^{-1}[M_0\cap B]\subseteq f^{-1}[M_0]\in\J\restriction A$ as $M_0\in\fin$ and $f^{-1}[\{i\}]\in\J\restriction A$ for all $i\in\omega$) and $k$ many sets $C$ satisfying the property $|C\cap M_n|\leq n$ for all $n>0$. Thus, to finish the proof we need to show that for each set $C\subseteq\omega$ such that $|C\cap M_n|\leq n$ for all $n$ we have $f^{-1}[C]\in\J\restriction A$. 

Fix $C\subseteq\omega$ as above and let $C_n=C\cap (M_{n+1}\setminus M_{n})$ for all $n\in\omega$. Then for each $n\in\omega$ we have $|C_{2n+2}|\leq 2n+3$ and 
$$C_{2n+2}\cap \$(1,C_0,3,C_2,\ldots,2n+1,C_{2n},2n+3)\subseteq C_{2n+2}\cap M_{2n+2}=\emptyset.$$
Thus, $C'=\bigcup_{n\in\omega}C_{2n}$ is an outcome of a play of the game along the winning strategy for Player II. Hence, $f^{-1}[C']\in\J\restriction A$. Similarly one can show that $f^{-1}[C'']\in\J\restriction A$ where $C''=\bigcup_{n\in\omega}C_{2n+1}=C\setminus C'$.
 
(\ref{lem:gra:determined-fixed-set}) We will use Martin's theorem on Borel determinacy. Indeed, using the notation from \cite[Section 20A]{MR1321597}, $G(\J,A,f)$ is equivalent to the game $G(T,\phi^{-1}[\J\restriction A])$ where 
$$T=\left\{(F_0,\ldots,F_n)\in\fin^{<\omega}:\forall {i\in \omega}\left(|F_{3i}|=1\text{ and }F_{3i+1}\cap F_{3i+2}=\emptyset\text{ and }\right.\right.$$
$$\left.\left.\left(|F_{3i+2}|\leq x\text{, where }\{x\}=F_{3i}\right)\right)\right\}$$
and $\phi:[T]\to\cP(A)$ given by $\phi((F_0,F_1,\ldots))=f^{-1}[\bigcup_{i\in\omega} F_{3i+2}]$ is a function of Baire class $1$ ($[T]\subseteq\fin^\omega$ denotes the set of all infinite branches of the tree $T$).

(\ref{lem:gra:determined-all-sets}) Let $A\notin\J$ and $f:A\to\omega$ be such that $f^{-1}[\{n\}]\in\J$ for all $n\in\omega$. Since $\J\restriction A$ is not $\leq_{\Kat}$-above any tall summable ideal, Player II cannot have a winning strategy in $G(\J,A,f)$ by item (\ref{lem:gra:win-strategy-for-II}). However, the game $G(\J,A,f)$ is determined by item (\ref{lem:gra:determined-fixed-set}) as $\J$ being Borel implies that $\J\restriction A$ is Borel as well. Thus, Player I has a winning strategy in $G(\J,A,f)$.
\end{proof}

The following immediate corollary seems to be interesting in itself as the property of containing a tall summable ideal is quite natural in the studies of ideals on $\omega$.

\begin{corollary}
	Let $\J$ be a Borel ideal. Then $\J$ does not contain a tall summable ideal if and only if 
	there are families $\cF_0,\cF_1,\ldots\subseteq \fin$
	such that
	\begin{enumerate}
		\item $\sup \{|F|: F\in \cF_n\}< \infty$ for each $n$,
		\item $\sup \{\min(F): F\in\cF_n\}=\infty$ for each $n$,
		\item if $F_n\in\cF_n$ for each $n$, then $\bigcup_{n\in\omega} F_n \notin\J$.
	\end{enumerate}	
\end{corollary}

\begin{proof}
	It follows from Lemma \ref{lem:gra} applied to the game $G(\J,\omega,\text{id})$.
\end{proof}

At this point we turn our attention to providing examples of ideals for which Player I has a winning strategy in $G(\J,A,f)$ for all $A$ and $f$.

\begin{proposition}
	\label{prop:ideals-not-above-summable}	
If $\J\in \{\fin^2, \ED_\fin,\ED,\conv\}$, then  
the ideal $\J\restriction A$ is not $\leq_{\Kat}$-above any tall summable ideal for every $A\notin\J$ (in particular, Player I has a winning strategy in $G(\J,A,f)$ for all $A$ and $f$).
\end{proposition}

\begin{proof}
For $\J\in \{\fin^2,\ED_{\fin}\}$, it follows from Propositions~\ref{prop:ideals-not-below-abov-other-ideals-in-Katetov-order}(\ref{prop:P-ideal-not-K-below-FINxFIN:item}, \ref{prop:P-ideal-not-K-below-ED-FIN:item})
and \ref{prop:homogeneous-idelas}.

For $\J=\conv$, it follows from the facts that $\Fin^2$ has the required property and that $\conv\restriction A\leq_{\Kat}\Fin^2$ for every $A\notin\conv$ (by \cite{MR3696069} for each $A\notin\conv$ one can find $B\notin\conv$, $B\subseteq A$ such that $\conv\restriction B$ is isomorphic to $\Fin^2$ and this particular isomorphism is a witness for $\conv\restriction A\leq_{\Kat}\Fin^2$). 

For $\J=\ED$, it follows from the facts that $\ED_{\fin}$ has the required property and that $\ED\restriction A\leq_{\Kat}\ED_{\fin}$ for every $A\notin\ED$ (as $A\notin\ED$ allows to construct inductively an increasing sequence $(m_n)$ such that $|A_{(m_n)}|\geq n$ and using this sequence one can define an injective function $f:\Delta\to A$ satisfying $f(n,k)\in\{m_n\}\times A_{(m_n)}$, for each $(n,k)\in\Delta$, which witnesses $\ED\restriction A\leq_{\Kat}\ED_{\fin}$).

The ``in particular'' part follows from Lemma~\ref{lem:gra}(\ref{lem:gra:determined-all-sets}).
\end{proof}

\begin{proposition}
	\label{prop:star}	
	Suppose that $\J=\fin(\phi)$ is a $\bf{\Sigma^0_2}$ ideal such that:
	$$(\star)\ \forall_{n\in\omega}\ \exists_{k\in\omega}\ \forall_{A\notin\fin(\phi)}\ \exists_{F\in\fin}\ (|F|\leq k\land \phi(F\cap A)\geq n).$$
	Then for each $A\notin\J$ the ideal $\J\restriction A$ is not $\leq_{\Kat}$-above any tall summable ideal (in particular, Player I has a winning strategy in $G(\J,A,f)$ for all $A$ and $f$).
\end{proposition}

\begin{proof}	
	The "in particular" part follows from Lemma \ref{lem:gra}(\ref{lem:gra:determined-all-sets}).

	Let $A\notin\J$, $f:A\to\omega$ and $\I_g$ be a tall summable ideal. If $f^{-1}[\{n\}]\notin\J$ for some $n\in\omega$ then we are done, so suppose that $f^{-1}[\{n\}]\in\J$ for all $n\in\omega$.

	We will inductively pick finite sets $B_n$, for $n\in\omega$, such that:
	\begin{itemize}
		\item $\phi(f^{-1}[B_n])\geq n$;
		\item $\sum_{i\in B_n}g(i)\leq \frac{1}{2^n}$.
	\end{itemize}

	In the $n$th step, using condition $(\star)$ find $k\in\omega$ such that:
	$$\forall_{C\notin\fin(\phi)}\ \exists_{F\in\fin}\ (|F|\leq k\land\phi(F\cap C)\geq n).$$

	Since $\I_g$ is tall, $\lim_{i\to \infty}g(i)=0$. Hence, there is $G\in\fin$ such that $g(i)\leq\frac{1}{k\cdot 2^n}$ for all $i\in\omega\setminus G$. Then $A\setminus f^{-1}[G]\notin\J$ (as $A\notin\J$ and $f^{-1}[G]\in\J$), so there is $F_n\in\fin$ such that $|F_n|\leq k$ and $\phi(F_n\cap (A\setminus f^{-1}[G]))\geq n$. Put $B_n=f[F_n]$ and observe that $\sum_{i\in B_n}g(i)\leq k\cdot \frac{1}{k\cdot 2^n}=\frac{1}{2^n}$.

	Define $B=\bigcup_{n\in\omega}B_n$. Then $\phi(f^{-1}[B])\geq \phi(f^{-1}[B_n])\geq n$ for all $n\in \omega$, so $f^{-1}[B]\notin\J\restriction A$. On the other hand, $\sum_{i\in B}g(i)\leq\sum_{n\in\omega}\sum_{i\in B_n}g(i)\leq\sum_{n\in\omega}\frac{1}{2^n}<\infty$, so $B\in \I_g$.	
\end{proof}

\begin{proposition}\ 
	\label{prop:star:examples}	
The ideals $\fin$, $\ED$, $\ED_\fin$ and $\cW$ all are examples of $\bf{\Sigma^0_2}$ ideals satisfying condition $(\star)$ from Proposition \ref{prop:star}. 
\end{proposition}

\begin{proof}
We will show that $\cW$ and $\ED$ satisfy condition $(\star)$ from Proposition~\ref{prop:star}. The cases of $\fin$ and $\ED_\fin$ are similar.

Observe that $\cW=\fin(\phi_{\cW})$ where $$\phi_{\cW}(A)=\sup\{n\in\omega:A\text{ contains an arithmetic progression of length }n\}$$ 
for all $A\subseteq\omega$. We claim that $k=n$ works for all $n$. Fix $n\in\omega$ and $A\notin\cW$. Then $A$ contains arbitrary long finite arithmetic progressions. In particular, there is an arithmetic progression $F\subseteq A$ of length $n$. Clearly, $|F|=n$ and $\phi_{\cW}(F\cap A)=\phi_{\cW}(F)=n$.

Similarly, $\ED=\fin(\phi_{\ED})$, where $\phi_{\ED}(A)=\min\{n\in\omega:\forall m\geq n(|A_{(m)}|\leq n)\}$ for all $A\subseteq\omega^2$ (see \cite{alcantara-phd-thesis}). Again, we claim that $k=n+1$ works for all $n$. Fix $n\in\omega$ and $A\notin\ED$. Then $\phi_{\ED}(A)>n$, so $|A_{(m)}|>n$ for some $m\geq n$. Let $F\subseteq A_{(m)}$ be any set of cardinality $n+1$. Then $|\{m\}\times F|=|F|=n+1$ and $\phi_{\ED}((\{m\}\times F)\cap A)=\phi_{\ED}(\{m\}\times F)\geq n+1$, so $\{m\}\times F$ is the required set. 
\end{proof}

We are ready for the main result of this section.

\begin{theorem}
	\label{thm:wsp-vs-cov-M}
If $\I$ contains a tall summable ideal and Player I has a winning strategy in $G(\J,A,f)$ for all $A$ and $f$ (in particular, if $\J$ is a Borel ideal such that $\J\restriction A$ is not $\leq_{\Kat}$-above any tall summable ideal for each $A\notin\J$),
then 
	$$\cov(\cM) \leq \wsp(\I,\J,\omega^\omega).$$
\end{theorem}

\begin{proof}
The ``in particular'' part follows from Lemma~\ref{lem:gra}(\ref{lem:gra:determined-all-sets}).

	Since $\I$ contains a tall summable ideal, there is 
	$g:\omega\to[0,+\infty)$ such that 
	$\lim_{i\to\infty}g(i)=0$
	and 
	$\I_g=\{A\subseteq\omega:\sum_{a\in A}g(a)<\infty\}\subseteq\I$ 
	
	Fix a family $\cA$ having $\J^+$-FIP and such that $|\cA|<\cov(\cM)$.
	Let $f\in \omega^\omega$. We need to show that $\{f^{-1}[B]\}\cup\cA$ has $\J^+$-FIP for some $B\in\I$.

	If there is $n\in \omega$ such that $\{f^{-1}[\{n\}]\}\cup\cA$ has $\J^+$-FIP, then the proof is finished in this case.
	
	Now, suppose that 
	$\{f^{-1}[\{n\}]\}\cup\cA$ does not have $\J^+$-FIP for any $n\in \omega$.
	
	We define a poset $(\poset, \leq)$ as follows:
	$$\mathbb{P}=\left\{p\in\fin\setminus\{\emptyset\} : \sum_{i\in p}g(i)\leq\left(2-\frac{1}{2^{|p|}}\right)\cdot \frac{1}{1+\min p}\right\}$$
	and 
	$q\leq p$ $\iff$
	\begin{enumerate}
		\item $q\supseteq p$, 
		\item 
		$\min(q\setminus p) > \max p$ [here $\min \emptyset = \omega$ by convention].
	\end{enumerate}
	
	Let 
	$$T=\left\{H\in[\cA]^{<\omega}: f^{-1}[\{n\}]\cap \bigcap H\notin\J \text{ for infinitely many $n$} \right\}.$$
	
	For each $H\in [\cA]^{<\omega}\setminus T$, there is $k_H\in\omega$ such that $f^{-1}[\{n\}] \cap \bigcap H\in \J$ for each $n\geq k_H$. Then, $\bigcap H\setminus f^{-1}[\{i:i<k_H\}] \notin \J$ (otherwise, $\{f^{-1}[\{n\}]\}\cup\cA$ would have $\J^+$-FIP for some $n<k_H$) and $f\restriction (\bigcap H\setminus f^{-1}[\{i:i<k_H\}])$ is $\J$-to-1. Since Player I has a winning strategy in $G(\J, \bigcap H\setminus f^{-1}[\{i:i<k_H\}]) ,f\restriction (\bigcap H\setminus f^{-1}[\{i:i<k_H\}]))$, we can use Lemma~\ref{lem:gra}(\ref{lem:gra:win-strategy-for-I}) to obtain 
	families $\cF_0^{H},\cF_1^{H},\dots\subseteq\fin$ such that:
	\begin{enumerate}
		\item $\sup \{|F|: F\in \cF_n^{H} \}< \infty$ for each $n$,
		\item $\sup \{\min(F): F\in\cF_n^{H}\}=\infty$ for each $n$,
		\item if $F_n\in\cF_n^{H}$ for each $n$, then $f^{-1}[\bigcup_{n\in\omega} F_n]\notin\J\restriction (\bigcap H\setminus f^{-1}[\{i:i<k_H\}]))$.
	\end{enumerate}

Let us define the following sets:
	\begin{enumerate}
		\item $D_{H} = \{p\in\poset : f^{-1}[p] \cap \bigcap H\notin\J\}$ 	for each $H\in T$,
		\item $E^n_{H}=\{p\in\poset : F\subseteq p \text{ for some } F\in \cF_n^{H}\}$
		for each $n\in\omega$ and $H\in [\cA]^{<\omega}\setminus T$.	
	\end{enumerate}
	
	Below, we show that these sets are dense in $(\poset,\leq)$.
	
	(1)
	Let $H\in T$ and $p\in\mathbb{P}$. Since $\lim_{i\to \infty}g(i)=0$, there is $t\in\omega$ such that 
	$$g(i)<\left(\frac{1}{2^{|p|}}-\frac{1}{2^{|p|+1}}\right)\cdot \frac{1}{1+\min p}$$ 
	for all $i>t$. 
	Since $H\in T$, there is  $n>\max(t,\max(p))$ with $f^{-1}[\{n\}]\cap \bigcap H\notin\J$. 
	Let $q= p\cup\{n\}$.
	Then $q\in\mathbb{P}$ as
	\begin{equation*}
		\begin{split}
			\sum_{i\in q}g(i) 
			&< 
			\sum_{i\in p}g(i)+\left(\frac{1}{2^{|p|}}-\frac{1}{2^{|p|+1}}\right)\cdot \frac{1}{1+\min p}\leq
			\\&
			\leq\left(2-\frac{1}{2^{|p|}}\right)\cdot\frac{1}{1+\min p}+\left(\frac{1}{2^{|p|}}-\frac{1}{2^{|p|+1}}\right)\cdot\frac{1}{1+\min p}
			\\&=
			\left(2-\frac{1}{2^{|p|+1}}\right)\cdot\frac{1}{1+\min p}
			=
			\left(2-\frac{1}{2^{|q|}}\right)\cdot\frac{1}{1+\min q}.
		\end{split}
	\end{equation*}
	Since $q\in D_{H}$ and   $q\leq p$, $D_{H}$ is dense in $(\mathbb{P},\leq)$.
	
	(2)	
	Let $n\in\omega$, $H\in [\cA]^{<\omega}\setminus T$ and $p\in\mathbb{P}$.
	Put $k_n = \sup \{|F|:F\in \cF_n^{H}\}$.
	Since $\lim_{i\to \infty}g(i)=0$, there is $t\in\omega$ such that
	$$g(i)<\frac{1}{k_n}\cdot \left(\frac{1}{2^{|p|}}-\frac{1}{2^{|p|+k_n}}\right)\cdot \frac{1}{1+\min p}$$ 
	for all $i>t$. 
	Since $\sup \{\min(F): F\in\cF_n^{H}\}=\infty$, 
	there is $F\in\cF_n^{H}$ with $\min(F)>\max(t,\max(p))$. 
	Let $q = p\cup F$. Then $q\in\mathbb{P}$ as
	\begin{equation*}
		\begin{split}
			\sum_{i\in q}g(i) 
			&< 
			\sum_{i\in p}g(i)+\sum_{i\in F} \frac{1}{k_n}\cdot \left(\frac{1}{2^{|p|}}-\frac{1}{2^{|p|+k_n}}\right)\cdot \frac{1}{1+\min p}\leq
			\\&
			\leq\left(2-\frac{1}{2^{|p|}}\right)\cdot\frac{1}{1+\min p}+\left(\frac{1}{2^{|p|}}-\frac{1}{2^{|p|+k_n}}\right)\cdot\frac{1}{1+\min p}
			\\&=
			\left(2-\frac{1}{2^{|p|+k_n}}\right)\cdot\frac{1}{1+\min p}
			\leq 
			\left(2-\frac{1}{2^{|q|}}\right)\cdot\frac{1}{1+\min q}.
		\end{split}
	\end{equation*}
	Since $q\in E^n_{H}$ and  $q\leq p$, $E_{H}^n$ is dense in $(\mathbb{P},\leq)$.

	Let $\cD$ be the family of all the sets of the form $D_{H}$ and $E_H^n$.

Since 
$\poset$ is countable and  
$|\cD|< \cov(\cM)$, there is a filter $\cG\subseteq \poset$ such that $\cG\cap D\neq\emptyset$ for each $D\in \cD$ (see e.g.~\cite[Theorem~7.13]{MR2768685})).
	
	We define $B = \bigcup \cG$.
	If we show that $B\in \I$ and $\{f^{-1}[B]\}\cup \cA$ has $\J^+$-FIP, then the proof will be finished in this case.

	Let $B=\{b_i:\ i\in\omega\}$ be the increasing enumeration of $B$. For each $i\in\omega$, there is $p_i\in \cG$ with $b_i\in p_i$. Since $\cG$ is a filter, without loss of generality we may assume that $p_0\subseteq p_1\subseteq \dots$. 
	Then
	$$\sum_{i\in\omega}g(b_i)=\lim_{i\to\infty}\sum_{j\in p_i}g(j)
	\leq \lim_{i\to\infty} \left(2-\frac{1}{2^{|p_i|}}\right)\cdot \frac{1}{1+\min  p_i}
	= \frac{2}{1+\min B}<\infty,$$
	so $B\in \I_g\subseteq\I$.

	Finally, we show that $\{f^{-1}[B]\}\cup \cA$ has $\J^+$-FIP.
	Let $H\in[\cA]^{<\omega}$. 
	
	If $H\in T$, take $p\in \cG\cap D_H$.
	Then $f^{-1}[B]\cap \bigcap H \supseteq f^{-1}[p]\cap \bigcap H\notin\J$.

	If $H \in[\cA]^{<\omega}\setminus T$, take  
	$p_n\in \cG\cap E_H^n$ for each $n\in\omega$, and pick $F_n\in \cF_n^{H}$ such that $F_n\subseteq p_n$.
	Then 
	$$f^{-1}[B]\cap \bigcap H 
	\supseteq 
	f^{-1}\left[\bigcup_{n\in\omega}p_n\right]\cap \bigcap H
	\supseteq 
	f^{-1}\left[\bigcup_{n\in\omega}F_n\right]\cap \bigcap H
	\notin\J.$$  
\end{proof}

\begin{corollary}
	\label{cor:covM-implies-exists-I-ultrafilter-not-J-ultrafilter}
	Assume $\cov(\cM)=\continuum$.
\begin{enumerate}
\item 
If $\I$ contains a tall summable ideal and Player I has a winning strategy in $G(\J,A,f)$ for all $A$ and $f$ (in particular, if $\J$ is a Borel ideal such that $\J\restriction A$ is not $\leq_{\Kat}$-above any tall summable ideal for each $A\notin\J$),
then 
there exists an $\I$-ultrafilter which is not a $\J$-point.	

	\item There exists an $\I$-ultrafilter 	for each tall analytic P-ideal ideal $\I$.
(It is an extension of \cite[Proposition~2.5.3]{flaskova-phd-thesis}, where the author proved under $\cov(\cM)=\continuum$ that there exists an $\I$-ultrafilter for each summable ideal.)

\item For each tall analytic P-ideal $\I$, there exists an  $\I$-ultrafilter which is not a $\cW$-point.	 
(It is an extension of \cite[Proposition~2.5.2]{flaskova-phd-thesis}, where the author proved under $\cov(\cM)=\continuum$ that there exists an $\I_{1/n}$-ultrafilter which is not a $\cW$-ultrafilter.)
\end{enumerate}
\end{corollary}

\begin{proof}
(1)	
 follows from Theorems~\ref{thm:wsp-vs-cov-M}
	and \ref{thm:wsp-implies-existence-of-I-ultrafilters}.
(2)	follows from (1) and 
Propositions~\ref{prop:tall-analytic-P-ideal-contains-summable-ideal}.
(3) follows from (1), Propositions~\ref{prop:tall-analytic-P-ideal-contains-summable-ideal}, \ref{prop:star} and \ref{prop:star:examples}.
\end{proof}


\section{P-points which are not $\J$-ultrafilters}

Now we turn our attention to ideals $\J$ such that it is consistent that there is a P-point which is not a $\J$-ultrafilter. We start with an immediate consequence of our previous results.

\begin{proposition}
	\label{thm:P-poiints-not-J-ultrafilter-equivalence}
	Let $\J$ be a $\text{P}^+(\J)$-ideal which is $\leq_{K}$-homogeneous.	
	\begin{enumerate}
		\item Assume CH. The following conditions are equivalent.
		\begin{enumerate}
			\item $\fin^2\not\leq_{\Kat}\J$.	
			\item There is a P-point which is not a $\J$-point.	
			\item There is a P-point which is not a weak $\J$-ultrafilter.	
			\item There is a P-point which is not a $\J$-ultrafilter.	
		\end{enumerate}
		\item 
		If $\J$ is $F_\sigma$, then CH can be relaxed to the assumption $\pnumber=\continuum$.
	\end{enumerate}
\end{proposition}

\begin{proof}
	It follows from Theorems~\ref{thm:KAT-iff-EXISTS-ULTRA-for-K-uniform-Pplus-ideals}(\ref{thm:KAT-iff-EXISTS-ULTRA-for-K-uniform-Pplus-ideals:CH-KAT})
	and \ref{thm:known-ultrafilters-characterized-as-I-ultrafilters}(\ref{thm:known-ultrafilters-characterized-as-I-ultrafilters:P-point}). In the case of $F_\sigma$-ideals additionally we need to use Theorems~\ref{thm:KAT-iff-EXISTS-ULTRA-for-K-uniform-Pplus-ideals}(4).
\end{proof}

Although the assumptions on $\J$ in the above proposition seem to be strong, it can be applied for instance for $\J=\ED_{\fin}$ and $\J=\cW$ (by  Propositions~\ref{prop:examples-of-P-plus-ideals}(\ref{prop:examples-of-P-plus-ideals:F-sigma}), \ref{prop:homogeneous-idelas} and \ref{prop:P-plus-ideal-not-K-above-FINxFIN}(\ref{prop:P-plus-ideal-not-K-above-FINxFIN:item})).
However, we cannot drop the assumption on the ideal $\J$ in general.
Indeed,  it is known that $\fin^2\not\leq_{\Kat}\conv$ (\cite{MR3696069}), but 
all P-points (i.e.~$\fin^2$-ultrafilters)
are $\conv$-ultrafilters by Theorem~\ref{thm:known-ultrafilters-characterized-as-I-ultrafilters}(\ref{thm:known-ultrafilters-characterized-as-I-ultrafilters:P-point}). 
Nevertheless, in Theorem~\ref{thm:new-order-above-FINxFIN-iff-extendable-to-Pplus} we show that 
the assumption on the ideal $\J$ can be removed if we replace the Kat\v{e}tov order $\leq_{\Kat}$ by the property that $\J$ can be extended to a $P^+$-ideal. In order to show it, we will need the following rather easy result about $P^+$-ideals.

\begin{lemma}\label{lem:preservation-of-P+-ideals}
Let $\I$ and $\J$ be ideals on $\omega$.
	\begin{enumerate}
		\item 
	If $\I\leq_{\Kat}\J$ and $\J$ is a $P^+$-ideal, then $\I$ can be extended to a $P^+$-ideal.	\label{lem:preservation-of-P+-ideals:by-Kat}

\item
If $\J$ is a $P^+$-ideal and $\cA\subseteq\cP(\omega)$ is a countable family such that $\omega\notin\J(\cA)$, then $\J(\cA)$ is a $P^+$-ideal.	\label{lem:preservation-of-P+-ideals:by-extension}

\item
If $\J$ is not a $P^+$-ideal, then there is a countable family $\cA\subseteq\cP(\omega)$  such that $\omega\notin\J(\cA)$ and $\fin^2\leq_{\Kat} \J(\cA)$.\label{lem:preservation-of-P+-ideals:non-P+}
	\end{enumerate}
\end{lemma}

\begin{proof}
	(\ref{lem:preservation-of-P+-ideals:by-Kat})
Let $f:\omega\to\omega$ be such that $f^{-1}[B]\in \J$ for each $B\in \I$.
Let $\cK=\{B: f^{-1}[B]\in \J\}$.
It is easy to see that $\cK$ is an ideal and $\I\subseteq\cK$.
We show that $\cK$ is a $P^+$-ideal.
Let $(B_n)\subseteq\cK^+$ and $B_n\supseteq B_{n+1}$ for each $n$.
Then $f^{-1}[B_n]\in \J^+$ and 
$f^{-1}[B_n]\supseteq f^{-1}[B_{n+1}]$ for each $n$.
Since $\J$ is a $P^+$-ideal, there is $A\notin \J$ such that $A\setminus f^{-1}[B_n]$ is finite for each $n$.
Then $B=f[A]\notin\cK$ and 
$B\setminus B_n$ is finite for each $n$.

(\ref{lem:preservation-of-P+-ideals:by-extension})
Let $\cA=\{A_n:n\in\omega\}$ be such that $\omega\notin\J(\cA)$.
Take any decreasing sequence $\{B_n:n\in\omega\}\subseteq \J(\cA)^+$ and define $C_0=\omega$ and $C_{n+1}=B_n\setminus \bigcup_{i<n}A_i$. Then $C_{n}\notin\J(\cA)$ and $C_n\supseteq C_{n+1}$ for each $n$.
Since $\J$ is a $P^+$-ideal, there is $C'\notin\J$ such that $C'\setminus C_n$ is finite for each $n$. Define $C=C'\cap C_1$ and note that $C\notin\J$ as $C'\setminus C=C'\setminus C_1$ is finite.
Once we show that $C\notin\J(\cA)$, the proof will be finished.
Suppose that $C\in \J(\cA)$. Then there is $n$ such that $C\setminus \bigcup_{i<n}A_i\in \J$.
Moreover, $C\cap  \bigcup_{i<n}A_i$ is finite since $C\cap  \bigcup_{i<n}A_i \subseteq C\setminus C_{n+1}$.
Thus, $C\in \J$, a contradiction.

(\ref{lem:preservation-of-P+-ideals:non-P+})
Since $\J$ is not a $P^+$-ideal, there is a decreasing sequence $(B_n)$ of sets not belonging to $\J$ such that if $X\setminus B_n\in\fin$ for all $n$, then $X\in\J$. 
Let $A_0=\omega\setminus B_0$ and $A_{n+1}=B_n\setminus B_{n+1}$ for all $n$. Observe that $\bigcup_{i\leq n}A_i=\omega\setminus B_n \notin\J^*$, so $\omega\notin \J(\{A_n:n\in\omega\})$. What is more, any one-to-one function $f:\omega\to\omega^2$ such that $f[A_n]\subseteq\{n\}\times\omega$ witnesses $\fin^2\leq_{\Kat}\J(\{A_n:n\in\omega\})$. 
\end{proof}

\begin{theorem}\label{thm:new-order-above-FINxFIN-iff-extendable-to-Pplus}
Let $\J$ be an ideal on $\omega$.
	\begin{enumerate}
		\item 	The following conditions are equivalent.	\label{thm:new-order-above-FINxFIN-iff-extendable-to-Pplus:ZFC}
	\begin{enumerate}
		\item $\J$ is extendable to a $P^+$-ideal.	\label{thm:new-order-above-FINxFIN-iff-extendable-to-Pplus:extendability}
		
\item $\fin^2\not\preceq_{\omega^\omega,1-1}^{\omega_1}\J$.	\label{thm:new-order-above-FINxFIN-iff-extendable-to-Pplus:new-order}	

\item $\fin^2\not\preceq_{\omega^\omega,\fin-1}^{\omega_1}\J$.	\label{thm:new-order-above-FINxFIN-iff-extendable-to-Pplus:new-order-KB}

\item $\fin^2\not\preceq_{\omega^\omega,\omega^\omega}^{\omega_1}\J$.	\label{thm:new-order-above-FINxFIN-iff-extendable-to-Pplus:new-order-K}

\end{enumerate}

\item 
	Under CH, the above conditions are equivalent to the following conditions.\label{thm:new-order-above-FINxFIN-iff-extendable-to-Pplus:CH}
\begin{enumerate}

\setcounter{enumii}{4}

	\item There is a P-point which is not a $\J$-point.	\label{thm:new-order-above-FINxFIN-iff-extendable-to-Pplus:ultrafilters}

	\item There is a P-point which is not a weak $\J$-ultrafilter.	\label{thm:new-order-above-FINxFIN-iff-extendable-to-Pplus:ultrafilters-KB}

	\item There is a P-point which is not a $\J$-ultrafilter.	\label{thm:new-order-above-FINxFIN-iff-extendable-to-Pplus:ultrafilters-K}

\end{enumerate}
\item 
	If $\J$ is Borel, then CH can be relaxed to the assumption $\pnumber=\continuum$ in (\ref{thm:new-order-above-FINxFIN-iff-extendable-to-Pplus:CH}).\label{thm:new-order-above-FINxFIN-iff-extendable-to-Pplus:Borel}

\item If $\J$ is Borel, then the following conditions are equivalent.	\label{thm:new-order-above-FINxFIN-iff-extendable-to-Pplus:CON}
\begin{enumerate}
	\item It is consistent that there exists a P-point which is not a $\J$-point (weak $\J$-ultrafilter or $\J$-ultrafilter, resp.).\label{thm:new-order-above-FINxFIN-iff-extendable-to-Pplus:CON-CON}
	\item Under CH, there exists a P-point which is not a $\J$-point (weak $\J$-ultrafilter or $\J$-ultrafilter, resp.).\label{thm:new-order-above-FINxFIN-iff-extendable-to-Pplus:CON-CH}
\end{enumerate}
In particular, if one can find a P-point which is not a $\J$-point (weak $\J$-ultrafilter or $\J$-ultrafilter, resp.) in some sophisticated model of ZFC, then one could do it already under CH.
\end{enumerate}
\end{theorem}

\begin{proof}
	(\ref{thm:new-order-above-FINxFIN-iff-extendable-to-Pplus:extendability})$\implies$(\ref{thm:new-order-above-FINxFIN-iff-extendable-to-Pplus:new-order}) 
Let $\J'\supseteq \J$ and $\J'$ be a $P^+$-ideal.
By Lemma~\ref{lem:preservation-of-P+-ideals}(\ref{lem:preservation-of-P+-ideals:by-extension}), $\J'(\cA)$ is a $P^+$-ideal for each countable family $\cA\subseteq\cP(\omega)$ such that $\omega\notin\J'(\cA)$.
So $\fin^2\not\leq_{\Kat} \J'(\cA)$, by Proposition~\ref{prop:P-plus-ideal-not-K-above-FINxFIN}(\ref{prop:P-plus-ideal-not-K-above-FINxFIN:item}).
Thus, $\fin^2\not\preceq_{\omega^\omega,1-1}^{\omega_1}\J$.

The implications (\ref{thm:new-order-above-FINxFIN-iff-extendable-to-Pplus:new-order})$\implies$(\ref{thm:new-order-above-FINxFIN-iff-extendable-to-Pplus:new-order-KB}) 
and
(\ref{thm:new-order-above-FINxFIN-iff-extendable-to-Pplus:new-order-KB})$\implies$(\ref{thm:new-order-above-FINxFIN-iff-extendable-to-Pplus:new-order-K}) 
are obvious.

(\ref{thm:new-order-above-FINxFIN-iff-extendable-to-Pplus:new-order-K})$\implies$(\ref{thm:new-order-above-FINxFIN-iff-extendable-to-Pplus:extendability}) 
	Since $\fin^2\not\preceq_{\omega^\omega,\omega^\omega}^{\omega_1}\J$, there is an ideal $\J'\geq_{\Kat} \J$ such that $\fin^2\not\leq_{\Kat}  \J'(\cA)$ for each countable $\cA\subseteq\cP(\omega)$ such that $\omega\notin\J'(\cA)$. 
By Lemma~\ref{lem:preservation-of-P+-ideals}(\ref{lem:preservation-of-P+-ideals:non-P+}), $\J'$ is a $P^+$-ideal, hence by Lemma~\ref{lem:preservation-of-P+-ideals}(\ref{lem:preservation-of-P+-ideals:by-Kat}), $\J$ can be extended to a $P^+$-ideal.

(\ref{thm:new-order-above-FINxFIN-iff-extendable-to-Pplus:new-order})$\implies$ (\ref{thm:new-order-above-FINxFIN-iff-extendable-to-Pplus:ultrafilters}) 
Under CH, it  follows from Corollary~\ref{cor:new-order-iff-existence-of-ultrafilters}(\ref{cor:new-order-iff-existence-of-ultrafilters:Katetov}) and Theorem ~\ref{thm:known-ultrafilters-characterized-as-I-ultrafilters}.

The implications 
(\ref{thm:new-order-above-FINxFIN-iff-extendable-to-Pplus:ultrafilters})$\implies$ (\ref{thm:new-order-above-FINxFIN-iff-extendable-to-Pplus:ultrafilters-KB}) 
and
(\ref{thm:new-order-above-FINxFIN-iff-extendable-to-Pplus:ultrafilters-KB})$\implies$ (\ref{thm:new-order-above-FINxFIN-iff-extendable-to-Pplus:ultrafilters-K}) 
are obvious (and hold in ZFC).

(\ref{thm:new-order-above-FINxFIN-iff-extendable-to-Pplus:ultrafilters-K})$\implies$ (\ref{thm:new-order-above-FINxFIN-iff-extendable-to-Pplus:extendability})
Take a P-point $\cU$ which is not a $\J$-ultrafilter (i.e.~$\J\leq_{\Kat}\cU^*$).
Then $\cU^*$ is a $P^+$-ideal, so $\J$ can be extended to a $P^+$-ideal by Lemma~\ref{lem:preservation-of-P+-ideals}(\ref{lem:preservation-of-P+-ideals:by-Kat}).
(Note that this argument works in ZFC.)

(\ref{thm:new-order-above-FINxFIN-iff-extendable-to-Pplus:Borel}) 
Taking into account that only the proof of the implication 
``(\ref{thm:new-order-above-FINxFIN-iff-extendable-to-Pplus:new-order})$\implies$ (\ref{thm:new-order-above-FINxFIN-iff-extendable-to-Pplus:ultrafilters})'' 
was not done in ZFC, and (\ref{thm:new-order-above-FINxFIN-iff-extendable-to-Pplus:extendability})$\iff$(\ref{thm:new-order-above-FINxFIN-iff-extendable-to-Pplus:new-order}) holds in ZFC, 
we only need to show 
the implication 
``(\ref{thm:new-order-above-FINxFIN-iff-extendable-to-Pplus:extendability})$\implies$(\ref{thm:new-order-above-FINxFIN-iff-extendable-to-Pplus:ultrafilters})'' 
under the assumption $\pnumber=\continuum$ with $\J$ being Borel.

Assume that $\pnumber=\continuum$ and $\J$ is a Borel ideal that can be extended to a $P^+$-ideal.
It is known that if a Borel ideal can be extended to a $P^+$-ideal, then it can be extended to an $F_\sigma$-ideal (see e.g.~\cite[Theorem 3.2.7]{alcantara-phd-thesis} or ~\cite[Corollary 3.7]{MR3692233}).
	Let $\widehat{\J}$ be an $F_\sigma$ ideal such that $\J\subseteq \widehat{\J}$.
	Then $\widehat{\J}\restriction A$ is also an $F_\sigma$ ideal for each  $A\notin\widehat{\J}$.
	Thus, $\fin^2\not\leq_{\Kat}\widehat{\J}\restriction A$ for each  $A\notin\widehat{\J}$ by Proposition~\ref{prop:P-plus-ideal-not-K-above-FINxFIN}(\ref{prop:P-plus-ideal-not-K-above-FINxFIN:item}).
	By Theorems~\ref{thm:wsp-vs-p-number-and-addSTAR}, \ref{the:u-greater-than-wsp} and \ref{thm:p-number-for-F-sigma}, $\uu(\fin^2,\omega^\omega,\widehat{\J},1-1)>\continuum=|\omega^\omega|$.
	By Theorem~\ref{thm:unumber-iff-I-ultra-not-J-ultra}, we obtain a $\fin^2$-ultrafilter (hence a P-point by Theorem~\ref{thm:known-ultrafilters-characterized-as-I-ultrafilters}(\ref{thm:known-ultrafilters-characterized-as-I-ultrafilters:P-point})) which is not a $\widehat{\J}$-point (so it is not a $\J$-point as well). 

(\ref{thm:new-order-above-FINxFIN-iff-extendable-to-Pplus:CON})
The implication ``(\ref{thm:new-order-above-FINxFIN-iff-extendable-to-Pplus:CON-CH})$\implies$(\ref{thm:new-order-above-FINxFIN-iff-extendable-to-Pplus:CON-CON})'' is obvious, so below we show only the implication ``(\ref{thm:new-order-above-FINxFIN-iff-extendable-to-Pplus:CON-CON})$\implies$(\ref{thm:new-order-above-FINxFIN-iff-extendable-to-Pplus:CON-CH})''.
Assume that it is consistent that there exists a P-point  that is not a $\J$-point (weak $\J$-ultrafilter or $\J$-ultrafilter, resp.).
Then consistently $\uu(\fin^2,\omega^\omega,\J,1-1)=\continuum^+$ ($\uu(\fin^2,\omega^\omega,\J,Fin-1)=\continuum^+$ or $\uu(\fin^2,\omega^\omega,\J,\omega^\omega)=\continuum^+$, resp.) by  Theorem~\ref{thm:unumber-iff-I-ultra-not-J-ultra}, so in particular it is consistent that $\fin^2\not\preceq_{\omega^\omega,1-1}^{\omega_1}\J$ ($\fin^2\not\preceq_{\omega^\omega,Fin-1}^{\omega_1}\J$ or $\fin^2\not\preceq_{\omega^\omega,\omega^\omega}^{\omega_1}\J$, resp.).
Then, by item (\ref{thm:new-order-above-FINxFIN-iff-extendable-to-Pplus:ZFC}), it is consistent that $\J$ can be extended to a $P^+$-ideal, and since $\J$ is Borel, $\J$ can be extended to an $F_\sigma$-ideal (by \cite[Theorem~3.2.7]{alcantara-phd-thesis} or \cite[Corollary 3.7]{MR3692233}).
However, if $\J$ is a Borel ideal, then the sentence ``$\J$ can be extended to an $F_\sigma$-ideal'' is absolute, so $\J$ can be extended to an $F_\sigma$ ideal in ZFC.
Then, using (\ref{thm:new-order-above-FINxFIN-iff-extendable-to-Pplus:CH}) we obtain, under CH, 
that there exists a P-point which is not a $\J$-point (weak $\J$-ultrafilter or $\J$-ultrafilter, resp.).
\end{proof}

An ideal $\I$ has the \emph{property $\FinBW$} (in short: $\I\in \FinBW$) if for each bounded sequence $(x_n)_{n\in \omega}$ of reals there is $A\notin \I$ such that the subsequence $(x_n)_{n\in A}$ is convergent (\cite{MR2320288}).

The ideal 
$\I_d = \{A\subseteq \omega: \lim_{n\to \infty}|A\cap n|/n=0\}$ of asymptotic density zero sets,
the ideal 
$\I_u = \{A\subseteq \omega: \lim_{n\to \infty} \max\{|A\cap\{m,m+1,\ldots,m+n-1\}| : m\in \omega\}/n=0\}$ of Banach (a.k.a.~uniform) density zero sets, the ideal $\nwd$, the ideal $\conv$ and the ideal $\lacunary$ 
generated by lacunary sets (recall: a set $A\subseteq \omega$ is \emph{lacunary} (\cite{MR908265}) if $\lim_{n\to \infty}(a_{n+1}-a_n)=\infty$, where $(a_n)_{n\in \omega}$ is the increasing enumeration of all elements of $A$)
do not  have the $\FinBW$ property (see \cite[Example~3]{MR1181163}, \cite[Corollary~1]{MR2835960}, \cite[Example~3.1]{MR2320288}, \cite[Proposition~6.4]{MR3034318} and \cite[Lemma~3.2]{MR3640040}, respectively).

\begin{proposition}
	\label{prop:P-point-is-I-ultrafilter-for-non-finBW-ideals}
	If $\J\notin \FinBW$, then every P-point is a $\J$-ultrafilter.
	In particular, every P-point is an $\I_d$-ultrafilter, $\I_u$-ultrafilter, $\nwd$-ultrafilter, $\conv$-ultrafilter and $\lacunary$-ultrafilter (note that the cases of the ideals $\I_d$ and $\lacunary$ were earlier done in a different manner in \cite[Proposition~2.3.3]{flaskova-phd-thesis}).
\end{proposition}

\begin{proof}
	Let $\cU$ be a P-point and suppose that $\cU$ is not a $\J$-ultrafilter i.e.~$\J\leq_{\Kat}\cU^*$.
	Since $\J\notin\FinBW$, then $\conv\leq_{\Kat}\J$ (\cite[Section 2.7]{alcantara-phd-thesis}, see also \cite[Proposition~6.4]{MR3034318}).
	Consequently, $\conv\leq_{\Kat}\cU^*$, so $\cU$ is not a P-point by Theorem~\ref{thm:known-ultrafilters-characterized-as-I-ultrafilters}(\ref{thm:known-ultrafilters-characterized-as-I-ultrafilters:P-point}), a contradiction.	
\end{proof}

It is known (\cite[Proposition~4.1 and 3.4]{MR2320288}) that $\J\in \FinBW$ for every ideal $\J$ which can be extended to an $F_\sigma$ ideal. In the realm of analytic P-ideals,  $\J\in \FinBW$ is even equivalent to extendability of $\J$ to an $F_\sigma$ ideal (\cite[Proposition~6.5]{MR3034318}).
It was a question of Hru\v{s}ak (\cite[{Question~5.16}]{MR2777744}) whether the same holds in the realm of all Borel ideals, but the answer to his question is negative as was recently showed by Kwela (\cite{Kwela-question-of-Hrusak}).

\begin{theorem}
	\label{thm:wsp-vs-F-sigma}
	If $\J$ is an $F_\sigma$ ideal and $\I$ is not a weak P-ideal, then 
	$$\cov(\cM) \leq \wsp(\I,\J,\cF)$$
	for every $\cF \subseteq\omega^\omega$.
\end{theorem}

\begin{proof}
	Let $\phi$ be a lsc submeasure such that $\J=\fin(\phi)$.
	Let $\cA$ be a family having $\J^+$-FIP and such that $|\cA|<\cov(\cM)$. 
	Fix $f\in \cF$.
	
	Since $\I$ is not a weak P-ideal, there is a partition $\{B_n:n\in \omega\}\subseteq \I$ of $\omega$ such that for each $B\notin\I$ there is $n$ with $|B_n\cap B|=\omega$.
	For each $n\in \omega$, we define $C_n = f^{-1}[B_n]$.
	We have 2 cases.
	
	Case 1. 
	There is $n\in \omega$ such that $\bigcap F\cap C_n\notin \J$ for each finite $F\subseteq \cA$.
	Then, we put $C=B_n$.
	Since $C\in\I$ and $\cA\cup\{f^{-1}[C]\}$ has $\J^+$-FIP, the proof is finished in this case.
	
	Case 2.
	For each $n\in \omega$ 
	there exists a finite $F\subseteq \cA$ with
	$\bigcap F\cap C_n\in \J$.
	
	We define a poset $(\poset, \leq)$ as follows:
	$$\poset = \left\{(H,n)\in [\omega]^{<\omega}\times \omega: H\subseteq \bigcup_{i<n} C_i \right\}$$
	and $(H_1,n_1)\leq (H_2,n_2) \iff$
	\begin{enumerate}
		\item $H_1\supseteq H_2$,
		\item $n_1\geq n_2$,
		\item $\max H_2 < \min (H_1\setminus H_2)$ [here $\max\emptyset=-1$ and $\min\emptyset = \omega$ by convention],
		\item $(H_1\setminus H_2)\cap \bigcup_{i<n_2} C_i = \emptyset$.
	\end{enumerate}
	
	Now we define sets
	\begin{enumerate}
		\item $D_{F,N} = \{(H,n)\in \poset: \phi(\bigcap F\cap H)>N\}$ for each finite $F\subseteq\cA$ and $N\in \omega$,
		\item $E_k=\{(H,n)\in \poset: n>k\}$ for each $k\in \omega$
	\end{enumerate}
	and below show that these sets are dense in $(\poset,\leq)$.
	
	First we show that sets $D_{F,N}$ are dense.
	Take any $(H,n)\in \poset$.
	In the present case, we know that for each $i<n$ there is a finite $F_i\subseteq\cA$ with $\bigcap F_i\cap C_i\in \J$.
	Since $\cA$ has $\J^+$-FIP, $X = \bigcap_{i<n} \bigcap F_i\notin\J$.
	Consequently, there is a finite set $H'\subseteq X\setminus (\bigcup_{i<n}C_i \cup \max H)$
	with $\phi(H')>N$.
	We put $G=H\cup H'$
	and $m = \max(\{i\in\omega: G\cap C_i\neq \emptyset\}\cup\{n\})$.
	Since $(G,m)\in D_{F,N}$ and $(G,m)\leq (H,n)$, we obtain denseness of $D_{F,N}$.

	Now we show that $E_k$ are dense.
	Take any $(H,n)\in\poset$.
	Let $G=H$ and $m=\max(n,k+1)$. 
	Since $(G,m)\in E_{k}$ and $(G,m)\leq (H,n)$, we obtain denseness of $E_{k}$.

	Let $\cD$ be the family of all the sets of the form $D_{F,N}$ and $E_k$.

Since $\poset$ is countable and  $|\cD|< \cov(\cM)$, there is a filter 
$\cG\subseteq \poset$ such that $\cG\cap D\neq\emptyset$ for each $D\in \cD$ (see e.g.~\cite[Theorem~7.13]{MR2768685})).
	
	We define 
	$$
	B = f[A],
	\text{\ \ where\ \ }
	A = \bigcup\{H: (H,n)\in \cG \text{ for some $n\in \omega$}\}.$$
	If we show that $B\in\I$ and $\cA\cup\{f^{-1}[B]\}$ has $\J^+$-FIP, the proof will be finished.

	First we show that $\cA\cup\{f^{-1}[B]\}$ has $\J^+$-FIP.
	Take any finite set $F\subseteq\cA$.
	Let $N\in \omega$.
	Since $\cG$ is $\cD$-generic, there is $(H,n)\in \cG\cap D_{F,N}$.
	In particular, $H\subseteq A$ and $\phi(\bigcap F\cap H)>N$.
	Since $A\subseteq f^{-1}[B]$, we obtain $\phi(\bigcap F\cap f^{-1}[B])>N$.
	Consequently, $\phi(\bigcap F\cap f^{-1}[B])=\infty$, so $\bigcap F\cap f^{-1}[B]\notin\J$.

	Finally, we show that $B\in \I$.
	If we show  that $A\cap C_k$ is finite for each $k$, then $B\cap B_k$ will be finite for each $k$, and consequently $B\in \I$.
	
	Take any $k\in \omega$.
	Since $\cG$ is $\cD$-generic, there is $(H,n)\in \cG\cap E_k$.
	We claim that $A\cap C_k \subseteq H$ (in particular, it means that $A\cap C_k$ is finite).
	
	Take any $a\in A\cap C_k$.
	Then there is $(G,m)\in \cG$ such that $a\in G$.
	Since $\cG$ is a filter, there is $(K,l)\in \cG$ such that $(K,l)\leq (H,n)$ and $(K,l)\leq (G,m)$.
	We have 2 cases.
	
	Case 1.
	Assume that $m\leq k$.
	Since $(G,m)\in \poset$, $G\subseteq \bigcup_{i<m}C_i$, and consequently, $a\notin C_k$, a contradiction, so this case is not possible.

	Case 2.
	Assume that $m> k$.
	Since $(K,l)\leq (H,n)$, $(K\setminus H)\cap \bigcup_{i<n}C_i=\emptyset$.
	Since $(H,n)\in E_k$, $n>k$. Consequently, $(K\setminus H)\cap C_k=\emptyset$.
	Suppose, in the sake of contradiction, that $a\notin H$.
	Since $a\in G\subseteq K$, $a\in K\setminus H$.
	Since $a\in C_k$, $a\in (K\setminus H)\cap C_k$, a contradiction.
\end{proof}

The following corollary provides an extension of \cite[Theorem~3.1]{MR2512901}, where the author proved that under $\cov(\cM)=\continuum$, there exists a P-point which is not a $\J$-ultrafilter for each $F_\sigma$-ideal $\J$.

\begin{corollary}
	\label{cor:addM-implies-P-point-not-J-ultra-for-J-extended-to-Fsigma}
	Assume $\cov(\cM)=\continuum$.
	If an ideal $\J$ can be extended to an $F_\sigma$ ideal (in particular, if $\J$ is an analytic P-ideal with the property $\FinBW$) and $\I$ is not a weak P-ideal, then  there exists an $\I$-ultrafilter which is not a $\J$-point.	
	In particular, then there exists a P-point which is not a $\J$-point.
\end{corollary}

\begin{proof}
	Let $\cK$ be an $F_\sigma$ ideal with $\J\subseteq \cK$.
By Theorems~\ref{thm:wsp-implies-existence-of-I-ultrafilters} and \ref{thm:wsp-vs-F-sigma},
	there exists an ultrafilter $\cU$ such that $\I\not\leq_{\Kat}\cU^*$ and $\cK\subseteq \cU^*$.
	Then $\cJ\subseteq \cU^*$, so $\cU$ is not a $\J$-point.
	The first ``in particular'' part follows from 
the fact that any analytic P-ideal with the property $\FinBW$ can be extended to an $F_\sigma$-ideal (see~\cite[Theorem~4.2]{MR2320288}).
The second ``in particular'' part follows from an easy fact that the ideal $\fin^2$ is not a weak P-ideal, and not-so-easy fact that $\fin^2$-ultrafilters coincide with P-points (see Theorem~\ref{thm:known-ultrafilters-characterized-as-I-ultrafilters}(\ref{thm:known-ultrafilters-characterized-as-I-ultrafilters:P-point})).
\end{proof}

A set $A\subseteq \omega$ is \emph{thin} (\cite{MR1845008}) if $\lim_{n\to\infty}a_n/a_{n+1}=0$, where $(a_n)_{n\in \omega}$ is the increasing enumeration of elements of $A$. 
Using the ratio test for the convergence of a series, one can easily see 
that the ideal $\cT$ generated by thin sets is contained in the summable ideal $\I_{1/n} =\{A\subseteq \omega:\sum_{n\in A}1/n<\infty\}$.
Thus, using the above corollary, we see that there exists a P-point which is not a $\cT$-point (note that this particular instance of the above corollary in the case of the ideal $\cT$ was earlier proved in \cite[Theorem~3.4]{MR2512901}).


\section{$\I$-ultrafilters which are not P-points}

In this section we are interested in ideals $\I$ such that it is consistent that there is an $\I$-ultrafilter which is not a P-point. It will turn out that such ideals can be characterized using the Kat\v{e}tov order. However, we will start this section with some sufficient conditions for existence of $\I$-ultrafilter which is not a P-point. Results of this section combined with the ones from the previous section will give us a characterization of Borel ideals $\I$ for which the notions of P-point and $\I$-ultrafilter coincide.

\begin{theorem} 
	\label{thm:wsp-vs-non-P-point}
Let $\cF\subseteq\omega^\omega$ and $\I$ be an ideal on $\omega$.
	\begin{enumerate}
		\item 
		If $\I\not\leq_{\cF}\fin^2$ and $\cI$ is a tall ideal (in particular, if $\I$ is a tall P-ideal), then 
		$$\max\{\adds(\I),\omega_1\} \leq \wsp(\I,\fin^2, \cF).$$
		
		\item If $\I$ contains a tall summable ideal, then 
		$$\cov(\cM) \leq \wsp(\I,\fin^2, \cF).$$
		
		\item If $\I$ is a tall analytic P-ideal, then 
		$$\max(\adds(\I),\cov(\cM)) \leq \wsp(\I,\fin^2, \cF).$$
	\end{enumerate}	
\end{theorem}

\begin{proof}
	(1) follows from Theorems ~\ref{thm:pnumber-for-FINxFIN},~\ref{thm:wsp-vs-p-number-and-addSTAR} and Proposition~\ref{prop:homogeneous-idelas}.
	For the ``in particular'' part we additionally need Proposition~\ref{prop:P-ideal-not-K-below-FINxFIN}(\ref{prop:P-ideal-not-K-below-FINxFIN:item}).
	(2) follows from Theorem~\ref{thm:wsp-vs-cov-M} and Proposition~\ref{prop:ideals-not-above-summable}.	
	(3) follows from (1), (2) and Proposition~\ref{prop:tall-analytic-P-ideal-contains-summable-ideal}.
\end{proof}

\begin{corollary}\label{cor:wsp-vs-non-P-point}
Let $\I$ be an ideal on $\omega$.
	\begin{enumerate}
		\item 
		Assume $\max(\adds(\I),\omega_1)=\cc$.
		\begin{enumerate}
			\item 
			If $\I\not\leq_{\Kat} \fin^2$ ($\I\not\leq_{\KB}\fin^2$ or $\I\not\leq_{\point} \fin^2$, resp.), then there exists an $\I$-ultrafilter (weak $\I$-ultrafilter  or $\I$-point, resp.) which is not a P-point. 
			
			\item
			If $\I$ is a tall P-ideal, then there exists an $\I$-ultrafilter which is not a P-point. 
		\end{enumerate}
		
		\item Assume $\cov(\cM)=\cc$. If $\I$ contains a tall summable ideal,  then there exists an $\I$-ultrafilter which is not a P-point. 
		
		\item Assume $\max(\adds(\I),\cov(\cM))=\cc$. If $\I$ is a tall analytic P-ideal,  then there exists an $\I$-ultrafilter which is not a P-point. 
	\end{enumerate}	
\end{corollary}

\begin{proof}
	It follows from Theorems ~\ref{thm:pnumber-for-FINxFIN}, ~\ref{thm:wsp-vs-non-P-point},  \ref{thm:wsp-implies-existence-of-I-ultrafilters} and \ref{thm:known-ultrafilters-characterized-as-I-ultrafilters}(\ref{thm:known-ultrafilters-characterized-as-I-ultrafilters:P-point}).
\end{proof}

Next result characterizes ideals $\I$ such that it is consistent that there is an $\I$-ultrafilter which is not a P-point.

\begin{theorem}\label{thm:katetov-equivalent-I-ultrafilter-not-P-point}
Let $\I$ be an ideal on $\omega$.
	\begin{enumerate}
	\item Under CH, the following conditions are equivalent. \label{thm:katetov-equivalent-I-ultrafilter-not-P-point-1}
	\begin{enumerate}
		\item 
		$\I\not\leq_{\Kat} \fin^2$ ($\I\not\leq_{\KB} \fin^2$ or $\I\not\leq_{\point} \fin^2$, resp.). \label{thm:katetov-equivalent-I-ultrafilter-not-P-point:Katetov}
		\item 
		There exists an $\I$-ultrafilter (weak  $\I$-ultrafilter or $\I$-point, resp.) which is not a P-point.\label{thm:katetov-equivalent-I-ultrafilter-not-P-point:ultrafilter}
	\end{enumerate}

\item If $\I$ is a tall Borel ideal, then the following conditions are equivalent.	\label{thm:katetov-equivalent-I-ultrafilter-not-P-point:CON}
\begin{enumerate}
	\item It is consistent that there exists an $\I$-ultrafilter (weak $\I$-ultrafilter or $\I$-point, resp.) which is not a P-point.
	\item Under CH, there exists an $\I$-ultrafilter (weak $\I$-ultrafilter or $\I$-point, resp.) which is not a P-point.
\end{enumerate}
In particular, if one can find an $\I$-ultrafilter (weak $\I$-ultrafilter or $\I$-point, resp.) which is not a P-point in some sophisticated model of ZFC, then one could do it already under CH.
\end{enumerate}	
\end{theorem}

\begin{proof}	
	(\ref{thm:katetov-equivalent-I-ultrafilter-not-P-point:Katetov})
	$\implies$ 
	(\ref{thm:katetov-equivalent-I-ultrafilter-not-P-point:ultrafilter})
By Theorem~\ref{thm:pnumber-for-FINxFIN},
$\pnumber(\fin^2)=\continuum$ under CH, so Corollary~\ref{cor:wsp-vs-non-P-point} finishes the proof.

(\ref{thm:katetov-equivalent-I-ultrafilter-not-P-point:ultrafilter})
$\implies$ 
	(\ref{thm:katetov-equivalent-I-ultrafilter-not-P-point:Katetov})
	It follows from Theorem~\ref{thm:known-ultrafilters-characterized-as-I-ultrafilters}(\ref{thm:known-ultrafilters-characterized-as-I-ultrafilters:P-point}).

(\ref{thm:katetov-equivalent-I-ultrafilter-not-P-point:CON})
Follows from Theorems~\ref{thm:KAT-iff-EXISTS-ULTRA-for-K-uniform-Pplus-ideals}(\ref{thm:KAT-iff-EXISTS-ULTRA-for-K-uniform-Pplus-ideals:CON}), \ref{thm:known-ultrafilters-characterized-as-I-ultrafilters}(\ref{thm:known-ultrafilters-characterized-as-I-ultrafilters:P-point}) and Proposition~\ref{prop:homogeneous-idelas}.
\end{proof}

At this point we are already able to characterize Borel ideals $\I$ for which the notions of P-point and $\I$-ultrafilter coincide.

\begin{theorem} 
	Let $\I$ be a Borel ideal. The following conditions are equivalent:
	\begin{enumerate}
		\item 
		the notions of P-point and $\I$-ultrafilter coincide,
		\item 
		$\I\leq_K\fin^2$ and $\I$ is not extendable to an $F_\sigma$-ideal.
	\end{enumerate}
\end{theorem}

\begin{proof}
By \cite[Theorem~3.2.7]{alcantara-phd-thesis}, a Borel ideal $\I$ can be extended to a $P^+$-ideal iff it can be extended to an $F_\sigma$-ideal. Thus it suffices to apply Theorems \ref{thm:new-order-above-FINxFIN-iff-extendable-to-Pplus} and \ref{thm:katetov-equivalent-I-ultrafilter-not-P-point}.
\end{proof}

The ideals $\mathcal{BI}$ and $\conv(\I_d,(\frac{1}{2^{n+1}}))$ introduced in \cite{Unboring} and \cite{Kwela-question-of-Hrusak}, respectively, are examples of Borel ideals below the ideal $\fin^2$ in the Kat\v{e}tov order (by the definition of $\conv(\I_d,(\frac{1}{2^{n+1}}))$ and results of \cite{Unboring}, we have $\conv(\I_d,(\frac{1}{2^{n+1}}))\subseteq\conv\leq_K\mathcal{BI}\leq_K\fin^2$) and not extendable to an $F_\sigma$-ideal (this is shown in \cite{Kwela-question-of-Hrusak}). Thus, an ultrafilter $\cU$ is a P-point iff $\cU$ is a $\mathcal{BI}$-ultrafilter iff $\cU$ is a $\conv(\I_d,(\frac{1}{2^{n+1}}))$-ultrafilter.

Next two corollaries are not new, however we present them in order to show how they can be derived from our results.

\begin{corollary}[Folklore]
	Assume CH. 
	There exists a weak $\ED_{\fin}$-ultrafilter which is not a P-point (i.e.~there exists a Q-point which is not a P-point).	
\end{corollary}

\begin{proof}
	It follows from Theorem~\ref{thm:katetov-equivalent-I-ultrafilter-not-P-point}(\ref{thm:katetov-equivalent-I-ultrafilter-not-P-point-1})
	and Proposition~\ref{prop:ED-FIN-not-K-below-FINxFIN}(\ref{prop:ED-FIN-not-K-below-FINxFIN:item}).
\end{proof}

\begin{corollary}[{\cite[Proposition~2.1.14]{flaskova-phd-thesis}}]
	\label{thm:I-ultrafilter-not-P-point-for-tall-P-ideal}
	Assume CH. If $\I$ is a tall P-ideal, then there exists an $\I$-ultrafilter which is not a P-point.	
\end{corollary}

\begin{proof}
	It follows from Theorem~\ref{thm:katetov-equivalent-I-ultrafilter-not-P-point}
	and Proposition~\ref{prop:ED-FIN-not-K-below-FINxFIN}(5).
\end{proof}

In \cite[Theorem~3.2]{MR2512901} Fla\v{s}kov\'{a} strengthened Corollary~\ref{thm:I-ultrafilter-not-P-point-for-tall-P-ideal} by replacing CH with $\pnumber=\continuum$. However, the method  of the proof of her stronger result 
is different from the proofs we were trying to capture in the cardinal $\wsp(\I,\J,\cF)$.
 The following question can act as a test whether the cardinal $\wsp(\I,\J,\cF)$ can encompass more cases than it was cooked up for.

\begin{question}
	Does $\pnumber=\continuum$ imply 
	$\wsp(\I,\fin^2,\omega^\omega) =\continuum$ for each tall P-ideal?
\end{question}

On the other hand, our results provides another corollary which does not follow
 from the above mentioned result of Fla\v{s}kova.

\begin{corollary}
	If $\I$ is a tall P-ideal and $\adds(\I)=\continuum$, then there exists an $\I$-ultrafilter which is not a P-point.	
\end{corollary}

\begin{proof}
	It follows from Corollary~\ref{cor:wsp-vs-non-P-point} and Propositon ~\ref{prop:ED-FIN-not-K-below-FINxFIN}(\ref{prop:P-ideal-not-K-below-FINxFIN:item}).
\end{proof}


\section{Q-points and $\J$-ultrafilters}


This section is a counterpart of the previous two sections -- we study Q-points instead of P-points. 

\subsection{$\I$-ultrafilters which are not Q-points}

In the case of characterizing ideals $\I$ such that it is consistent that there is an $\I$-ultrafilter which is not a Q-point, again it will turn out that the Kat\v{e}tov order is a proper tool. We start with sufficient conditions for existence of $\I$-ultrafilter which is not a Q-point. 

\begin{theorem}\label{thm:wsp-vs-non-Q-point}
Let $\cF\subseteq\omega^\omega$ and $\I$ be an ideal on $\omega$.
\begin{enumerate}
	\item 
If $\I\not\leq_{\cF}\ED_{\fin}$ and $\I$ is a tall ideal (in particular, if $\I$ is a tall P-ideal), then 
		$$\max\{\adds(\I),\pnumber(\ED_{\fin})\} \leq \wsp(\I,\ED_{\fin}, \cF).$$

\item If $\I$ contains a tall summable ideal, then 
$$\cov(\cM) \leq \wsp(\I,\ED_{\fin}, \cF).$$

\item If $\I$ is a tall analytic P-ideal, then 
$$\max(\adds(\I),\pnumber(\ED_{\fin}),\cov(\cM)) \leq \wsp(\I,\ED_{\fin}, \cF).$$
\end{enumerate}	
\end{theorem}

\begin{proof}
(1) 
follows from Theorem~\ref{thm:wsp-vs-p-number-and-addSTAR}(\ref{thm:wsp-vs-p-number-and-addSTAR:pnumber}) and Proposition~\ref{prop:homogeneous-idelas}.
For the ``in particular'' part additionally one has to use Proposition~\ref{prop:P-ideal-not-K-below-ED-FIN}(\ref{prop:P-ideal-not-K-below-ED-FIN:item}).
(2)  
follows from Theorem~\ref{thm:wsp-vs-cov-M} and Proposition~\ref{prop:ideals-not-above-summable}.
(3) 
follows from (1), (2) and Proposition~\ref{prop:tall-analytic-P-ideal-contains-summable-ideal}.
\end{proof}

\begin{corollary}\label{cor:wsp-vs-non-Q-point}
Let $\I$ be an ideal on $\omega$.
\begin{enumerate}
	\item 
	Assume $\max(\adds(\I),\pnumber(\ED_{\fin}))=\cc$.
\begin{enumerate}
	\item 
If $\I\not\leq_{\Kat} \ED_{\fin}$ ($\I\not\leq_{\KB}\ED_{\fin}$ or $\I\not\leq_{\point}\ED_{\fin}$, resp.), then there exists an $\I$-ultrafilter (weak $\I$-ultrafilter  or $\I$-point, resp.) which is not a Q-point. 

\item
If $\I$ is a tall P-ideal, then there exists an $\I$-ultrafilter which is not a Q-point. 
\end{enumerate}

\item Assume $\cov(\cM)=\cc$. If $\I$ contains a tall summable ideal,  then there exists an $\I$-ultrafilter which is not a Q-point. 

\item Assume $\max(\adds(\I),\pnumber(\ED_{\fin}),\cov(\cM))=\cc$. If $\I$ is a tall analytic P-ideal,  then there exists an $\I$-ultrafilter which is not a Q-point. 
\end{enumerate}	
\end{corollary}

\begin{proof}
It follows from Theorems~\ref{thm:wsp-vs-non-Q-point},  \ref{thm:wsp-implies-existence-of-I-ultrafilters} and \ref{thm:known-ultrafilters-characterized-as-I-ultrafilters}(\ref{thm:known-ultrafilters-characterized-as-I-ultrafilters:Q-point}).
\end{proof}

Next result characterizes, with the use of Kat\v{e}tov order, ideals $\I$ such that it is consistent that there is an $\I$-ultrafilter which is not a Q-point.

\begin{theorem}\label{thm:katetov-equivalent-I-ultrafilter-not-Q-point}
Let $\I$ be an ideal on $\omega$.
	\begin{enumerate}
\item Under $\pnumber=\continuum$, the following conditions are equivalent.
	\begin{enumerate}
		\item 
		$\I\not\leq_{\Kat} \ED_{\fin}$ ($\I\not\leq_{\KB} \ED_{\fin}$ or $\I\not\leq_{\point}\ED_{\fin}$, resp.). \label{thm:katetov-equivalent-I-ultrafilter-not-Q-point:Katetov}
		\item 
		There exists an $\I$-ultrafilter (weak  $\I$-ultrafilter or $\I$-point, resp.) which is not a Q-point.\label{thm:katetov-equivalent-I-ultrafilter-not-Q-point:ultrafilter}
	\end{enumerate}	
\item If $\I$ is a tall Borel ideal, then the following conditions are equivalent.	\label{thm:katetov-equivalent-I-ultrafilter-not-Q-point:CON}
	
\begin{enumerate}
	\item It is consistent that there exists an $\I$-ultrafilter (weak $\I$-ultrafilter or $\I$-point, resp.) which is not a Q-point.
	\item Under CH, there exists an $\I$-ultrafilter (weak $\I$-ultrafilter or $\I$-point, resp.) which is not a Q-point.
\end{enumerate}
In particular, if one can find an $\I$-ultrafilter (weak $\I$-ultrafilter or $\I$-point, resp.) which is not a Q-point in some sophisticated model of ZFC, then one could do it already under CH.
\end{enumerate}
\end{theorem}

\begin{proof}	
	(\ref{thm:katetov-equivalent-I-ultrafilter-not-Q-point:Katetov})
	$\implies$ 
	(\ref{thm:katetov-equivalent-I-ultrafilter-not-Q-point:ultrafilter})
	Since the ideal $\ED_{\fin}$ is $F_\sigma$, $\pnumber(\ED_{\fin})\geq \pnumber$ by  Theorem~\ref{thm:p-number-for-F-sigma}. 
	Then Corollary~\ref{cor:wsp-vs-non-Q-point} finishes the proof.
	
	(\ref{thm:katetov-equivalent-I-ultrafilter-not-Q-point:ultrafilter}) $\implies$ (\ref{thm:katetov-equivalent-I-ultrafilter-not-Q-point:Katetov})
	It follows from Theorem~\ref{thm:known-ultrafilters-characterized-as-I-ultrafilters}(\ref{thm:known-ultrafilters-characterized-as-I-ultrafilters:Q-point}).

(2)
Follows from Theorems~\ref{thm:KAT-iff-EXISTS-ULTRA-for-K-uniform-Pplus-ideals}(\ref{thm:KAT-iff-EXISTS-ULTRA-for-K-uniform-Pplus-ideals:CON}) and \ref{thm:known-ultrafilters-characterized-as-I-ultrafilters}(\ref{thm:known-ultrafilters-characterized-as-I-ultrafilters:Q-point}) and Proposition~\ref{prop:homogeneous-idelas}.
\end{proof}

\begin{corollary}[Folklore]
	Assume $\pnumber=\continuum$. 
	There exists a $\Fin^2$-ultrafilter which is not a Q-point (i.e.~there exists a P-point which is not a Q-point).	
\end{corollary}

\begin{proof}
	It follows from Theorems~\ref{thm:katetov-equivalent-I-ultrafilter-not-Q-point} and \ref{thm:known-ultrafilters-characterized-as-I-ultrafilters}(\ref{thm:known-ultrafilters-characterized-as-I-ultrafilters:P-point})
and Proposition~\ref{prop:P-plus-ideal-not-K-above-FINxFIN}(\ref{prop:P-plus-ideal-not-K-above-FINxFIN:item}), because $\ED_{\fin}$ is an $F_\sigma$ ideal, so it is a $P^+$-ideal by Proposition~\ref{prop:examples-of-P-plus-ideals}(\ref{prop:examples-of-P-plus-ideals:F-sigma}).
\end{proof}


\subsection{Q-points which are not $\J$-ultrafilters}

\begin{theorem}
	\label{thm:Q-point-not-J-ultrafilter-equivalence}
		Let $\J$ be a $\text{P}^+(\J)$-ideal which is $\leq_{\KB}$-homogeneous.	
		\begin{enumerate}
\item Assume CH. The following conditions are equivalent.
		\begin{enumerate}
			\item $\ED_{\fin}\not\leq_{\KB}\J$.	
			\item There is a Q-point which is not a $\J$-point.	
			\item There is a Q-point which is not a weak $\J$-ultrafilter.	
		\end{enumerate}
		\item 
		If $\J$ is $F_{\sigma}$, then CH can be relaxed to the assumption $\pnumber=\continuum$.
	\end{enumerate}
\end{theorem}

\begin{proof}
It follows from Theorems~\ref{thm:KAT-iff-EXISTS-ULTRA-for-K-uniform-Pplus-ideals}(\ref{thm:KAT-iff-EXISTS-ULTRA-for-K-uniform-Pplus-ideals:CH-KB}) and \ref{thm:known-ultrafilters-characterized-as-I-ultrafilters}(\ref{thm:known-ultrafilters-characterized-as-I-ultrafilters:Q-point}). In the case of $F_\sigma$-ideals additionally we need to use Theorems~\ref{thm:KAT-iff-EXISTS-ULTRA-for-K-uniform-Pplus-ideals}(4).
\end{proof}

Q-points cannot be characterized as $\J$-ultrafilters for any ideal $\J$ -- to obtain a Q-point which is not a $\J$-ultrafilter one needs to assume nothing about an ideal  $\J$ as the following result of Fla\v{s}kova shows.

\begin{proposition}[{\cite[Proposition~2.4.7]{flaskova-phd-thesis}}]
	\label{prop:Q-point-not-J-ultraafilters}
	Assume $\cov(\cM)=\continuum$. For each ideal $\J$, there exists a Q-point which is not a $\J$-ultrafilter.
\end{proposition}


\section{Selective ultrafilters and $\J$-ultrafilters}


\subsection{Selective ultrafilters which are not $\J$-ultrafilters}

\begin{proposition}
	\label{prop:selective-ultrafilter-not-J-ultrafilter-equivalence}
	Let $\J$ be a $\text{P}^+(\J)$-ideal which is $\leq_{K}$-homogeneous.	
	\begin{enumerate}
		\item Assume CH. The following conditions are equivalent.
		\begin{enumerate}
			\item $\ED\not\leq_{\Kat}\J$.	
			\item There is a selective ultrafilter which is not a $\J$-point.	
			\item There is a selective ultrafilter which is not a weak $\J$-ultrafilter.	
			\item There is a selective ultrafilter which is not a $\J$-ultrafilter.	
		\end{enumerate}
		\item 
		If $\J$ is an $F_\sigma$-ideal, then CH can be relaxed to the assumption $\pnumber=\continuum$.
	\end{enumerate}
\end{proposition}

\begin{proof}
	(1) It follows from Theorems~\ref{thm:KAT-iff-EXISTS-ULTRA-for-K-uniform-Pplus-ideals}(\ref{thm:KAT-iff-EXISTS-ULTRA-for-K-uniform-Pplus-ideals:CH-KAT})
	and \ref{thm:known-ultrafilters-characterized-as-I-ultrafilters}(\ref{thm:known-ultrafilters-characterized-as-I-ultrafilters:Ramsey-point}). For (2) one additionally needs Theorems~\ref{thm:KAT-iff-EXISTS-ULTRA-for-K-uniform-Pplus-ideals}(4).
\end{proof}

In the above theorem, we cannot drop the assumptions on the ideal $\J$ in general.
Indeed,  it is known that 
$\ED\not\leq_{\Kat}\random$ (see \cite[Section 2]{MR3696069}), but 
all selective ultrafilters (i.e.~$\ED$-ultrafilters)
are $\random$-points (Theorem~\ref{thm:known-ultrafilters-characterized-as-I-ultrafilters}(\ref{thm:known-ultrafilters-characterized-as-I-ultrafilters:Ramsey-point})). 
Although, in Theorem~\ref{thm:new-order-above-ED-iff-extendable-to-selective}, we show that 
the assumption on the ideal $\J$ can be removed if we replace the Kat\v{e}tov order $\leq_{\Kat}$ by the property that $\J$ can be extended to a selective ideal.

\begin{lemma}\label{lem:preservation-of-selective-ideals}
Let $\I$ and $\J$ be ideals on $\omega$.
	\begin{enumerate}
		\item 
		If $\I\leq_{\Kat}\J$ and $\J$ is a selective ideal, then $\I$ can be extended to a selective ideal.	\label{lem:preservation-of-selective-ideals:by-Kat}
		
		\item
		If $\J$ is a selective ideal and $\cA\subseteq\cP(\omega)$ is a countable family such that $\omega\notin\J(\cA)$, then $\J(\cA)$ is a selective ideal.	\label{lem:preservation-of-selective-ideals:by-extension}
		
		\item
		If $\J$ is not a selective ideal, then there is a countable family $\cA\subseteq\cP(\omega)$  such that $\omega\notin\J(\cA)$ and $\ED\leq_{\Kat} \J(\cA)$.\label{lem:preservation-of-selective-ideals:non-selective}
	\end{enumerate}
\end{lemma}

\begin{proof}
	(\ref{lem:preservation-of-selective-ideals:by-Kat})
	Let $f:\omega\to\omega$ be such that $f^{-1}[B]\in \J$ for each $B\in \I$.
	Let $\cK=\{B: f^{-1}[B]\in \J\}$.
	It is easy to see that $\cK$ is an ideal and $\I\subseteq\cK$.
	We show that $\cK$ is a selective ideal.
	Let $(B_n)\subseteq \cK^+$ 	be a decreasing sequence.
	Then $f^{-1}[B_n]\in\J^+$ and $f^{-1}[B_n]\supseteq f^{-1}[B_{n+1}]$ for each $n$.
	Since $\J$ is a selective ideal, there is $A\in \J^+$ such that $|A\cap (f^{-1}[B_n]\setminus f^{-1}[B_{n+1}])|\leq 1$ for each $n$.
	Then $B=f[A]\in\cK^+$ and 
	$|B\cap (B_n \setminus B_{n+1})|\leq 1$ for each $n$.
	
	(\ref{lem:preservation-of-selective-ideals:by-extension})
	Let $\cA=\{A_n:n\in\omega\}$ be such that $\omega\notin\J(\cA)$.
Take any decreasing sequence $B_n\in\J(\cA)^+$.
Let $C_0=\omega$ and $C_{n+1} = B_n\setminus \bigcup_{i<n}A_i$.
Then $C_n\in \J(\cA)^+$ and $C_n\supseteq C_{n+1}$ for each $n$.
Since $C_n\in\J^+$ and  $\J$ is a selective ideal, there is $C'\in\J^+$ such that 
$|C'\cap (C_n\setminus C_{n+1})|\leq 1$ 
for each $n$.
Define $C=C'\cap C_1$. Then $C\in\J^+$ as $|C'\setminus C|=|C'\cap(C_0\setminus C_1)|\leq 1$.
Once we show that $C\in\J(\cA)^+$, the proof will be finished.
	Suppose that $C\in \J(\cA)$. Then there is $n$ such that $C\setminus \bigcup_{i<n}A_i\in \J$.
	But $C\cap  \bigcup_{i<n}A_i$ is finite (it has at most $n$ elements), so $C\in \J$, a contradiction.
	
	(\ref{lem:preservation-of-selective-ideals:non-selective})
Since $\J$ is not a selective ideal, there is a decreasing sequence $B_n\in\J^+$ such that if $|X\cap (B_n\setminus B_{n+1})|\leq 1$ for all $n$ then $X\in\J$. 
Let $A_0=\omega\setminus B_0$ and $A_{n+1}=B_n\setminus B_{n+1}$ for all $n$. Observe that $\bigcup_{i\leq n}A_i=\omega\setminus B_n \notin\J^*$, so $\omega\notin \J(\{A_n:n\in\omega\})$. What is more, any one-to-one function $f:\omega\to\omega^2$ such that $f[A_n]\subseteq\{n\}\times\omega$ witnesses $\ED \leq_{\Kat}\J(\{A_n:n\in\omega\})$. 
\end{proof}

\begin{theorem}\label{thm:new-order-above-ED-iff-extendable-to-selective}
Let $\J$ be an ideal on $\omega$.
	\begin{enumerate}
		\item 	The following conditions are equivalent.	\label{thm:new-order-above-ED-iff-extendable-to-selective:ZFC}
		\begin{enumerate}
			\item $\J$ is extendable to a selective ideal.	\label{thm:new-order-above-ED-iff-extendable-to-selective:extendability}

			\item $\ED\not\preceq_{\omega^\omega,1-1}^{\omega_1}\J$.	\label{thm:new-order-above-ED-iff-extendable-to-selective:new-order}

			\item $\ED\not\preceq_{\omega^\omega,\fin-1}^{\omega_1}\J$.	\label{thm:new-order-above-ED-iff-extendable-to-selective:new-order-KB}

			\item $\ED\not\preceq_{\omega^\omega,\omega^\omega}^{\omega_1}\J$.	\label{thm:new-order-above-ED-iff-extendable-to-selective:new-order-K}

		\end{enumerate}
		
		\item 
		Under CH, the above conditions are equivalent to the following condition.\label{thm:new-order-above-ED-iff-extendable-to-selective:CH}
		\begin{enumerate}
			
			\setcounter{enumii}{4}
			
		\item There is a selective ultrafilter which is not a $\J$-point.\label{thm:new-order-above-ED-iff-extendable-to-selective:ultrafilters}
	
		\item There is a selective ultrafilter which is not a weak $\J$-ultrafilter.	\label{thm:new-order-above-ED-iff-extendable-to-selective:ultrafilters-KB}

		\item There is a selective ultrafilter which is not a $\J$-ultrafilter.	\label{thm:new-order-above-ED-iff-extendable-to-selective:ultrafilters-K}

	\end{enumerate}
	\end{enumerate}
\end{theorem}

\begin{proof}
	(\ref{thm:new-order-above-ED-iff-extendable-to-selective:extendability})$\implies$(\ref{thm:new-order-above-ED-iff-extendable-to-selective:new-order}) 
	Let $\J'$ be a selective ideal such that $\J'\supseteq \J$.
	By Lemma~\ref{lem:preservation-of-selective-ideals}(\ref{lem:preservation-of-selective-ideals:by-extension}), $\J'(\cA)$ is a selective ideal for each countable family $\cA\subseteq\cP(\omega)$ such that $\omega\notin\J'(\cA)$.
	So $\ED\not\leq_{\Kat} \J'(\cA)$, by Proposition~\ref{prop:selective-ideal-not-K-above-ED}(\ref{prop:selective-ideal-not-K-above-ED:item}).
	Thus, $\ED\not\preceq_{\omega^\omega,1-1}^{\omega_1}\J$.

The implications (\ref{thm:new-order-above-ED-iff-extendable-to-selective:new-order})$\implies$(\ref{thm:new-order-above-ED-iff-extendable-to-selective:new-order-KB}) 
and
(\ref{thm:new-order-above-ED-iff-extendable-to-selective:new-order-KB})$\implies$(\ref{thm:new-order-above-ED-iff-extendable-to-selective:new-order-K}) 
are obvious.

	(\ref{thm:new-order-above-ED-iff-extendable-to-selective:new-order-K})$\implies$(\ref{thm:new-order-above-ED-iff-extendable-to-selective:extendability}) 
	Since $\ED\not\preceq_{\omega^\omega,\omega^\omega}^{\omega_1}\J$, there is an ideal $\J'\geq_{\Kat} \J$ such that $\ED\not\leq_{\Kat}  \J'(\cA)$ for each countable $\cA\subseteq\cP(\omega)$ such that $\omega\notin\J'(\cA)$. 
	By Lemma~\ref{lem:preservation-of-selective-ideals}(\ref{lem:preservation-of-selective-ideals:non-selective}), $\J'$ is a selective ideal, hence by Lemma~\ref{lem:preservation-of-selective-ideals}(\ref{lem:preservation-of-selective-ideals:by-Kat}), $\J$ can be extended to a selective ideal.
	
(\ref{thm:new-order-above-ED-iff-extendable-to-selective:new-order})$\implies$ (\ref{thm:new-order-above-ED-iff-extendable-to-selective:ultrafilters}) 
Under CH, it  follows from Corollary~\ref{cor:new-order-iff-existence-of-ultrafilters}(\ref{cor:new-order-iff-existence-of-ultrafilters:Katetov}) and Theorem ~\ref{thm:known-ultrafilters-characterized-as-I-ultrafilters}(3).

The implications (\ref{thm:new-order-above-ED-iff-extendable-to-selective:ultrafilters})$\implies$(\ref{thm:new-order-above-ED-iff-extendable-to-selective:ultrafilters-KB}) 
and
(\ref{thm:new-order-above-ED-iff-extendable-to-selective:ultrafilters-KB})$\implies$(\ref{thm:new-order-above-ED-iff-extendable-to-selective:ultrafilters-K}) 
are obvious (and hold in ZFC).

	(\ref{thm:new-order-above-ED-iff-extendable-to-selective:ultrafilters-K})$\implies$ (\ref{thm:new-order-above-ED-iff-extendable-to-selective:extendability})
	Take a selective ultrafilter $\cU$ which is not a $\J$-ultrafilter (i.e.~$\J\leq_{\Kat}\cU^*$).
	Then $\cU^*$ is a selective ideal, so $\J$ can be extended to a selective ideal by Lemma~\ref{lem:preservation-of-selective-ideals}(\ref{lem:preservation-of-selective-ideals:by-Kat}). (Note that this argument works in ZFC.)
\end{proof}


\subsection{$\I$-ultrafilters which are not selective}

\begin{theorem}
	\label{thm:wsp-vs-non-Ramsey-point}
If $\I$ contains a tall summable ideal, then 
		$$\cov(\cM) \leq \wsp(\I,\ED, \omega^\omega).$$
\end{theorem}

\begin{proof}
It follows from Theorem~\ref{thm:wsp-vs-cov-M} and Proposition~\ref{prop:ideals-not-above-summable}.
\end{proof}

\begin{corollary}\ 
	\label{cor:wsp-vs-non-Ramsey-point}
Assume $\cov(\cM)=\cc$. If $\I$ contains a tall summable ideal,  then there exists an $\I$-ultrafilter which is not a selective ultrafilter. 
\end{corollary}

\begin{proof}
	It follows from Theorems~\ref{thm:wsp-vs-non-Ramsey-point},  \ref{thm:wsp-implies-existence-of-I-ultrafilters} and \ref{thm:known-ultrafilters-characterized-as-I-ultrafilters}(\ref{thm:known-ultrafilters-characterized-as-I-ultrafilters:Ramsey-point}).
\end{proof}

The characterization of ideals $\I$ such that it is consistent that there exists an $\I$-ultrafilter which is not selective occurs to be a simple consequence of our previous considerations.

\begin{theorem}
	Let $\I$ be an ideal on $\omega$.
	\begin{enumerate}
	\item Under CH, the following conditions are equivalent.
	\begin{enumerate}
		\item 
		($\I\not\leq_{\Kat} \fin^2$ or $\I\not\leq_{\Kat} \ED_\fin$) (($\I\not\leq_{\KB} \fin^2$ or $\I\not\leq_{\KB} \ED_\fin$) or ($\I\not\leq_{P} \fin^2$ or $\I\not\leq_{P} \ED_\fin$), resp.). 
		\item 
		There exists an $\I$-ultrafilter (weak  $\I$-ultrafilter or $\I$-point, resp.) which is not selective.
	\end{enumerate}

\item If $\I$ is a tall Borel ideal, then the following conditions are equivalent.	
\begin{enumerate}
	\item It is consistent that there exists an $\I$-ultrafilter (weak $\I$-ultrafilter or $\I$-point, resp.) which is not selective.
	\item Under CH, there exists an $\I$-ultrafilter (weak $\I$-ultrafilter or $\I$-point, resp.) which is not selective.
\end{enumerate}
In particular, if one can find an $\I$-ultrafilter (weak $\I$-ultrafilter or $\I$-point, resp.) which is not selective in some sophisticated model of ZFC, then one could do it already under CH.
\end{enumerate}	
\end{theorem}

\begin{proof}	
Follows from Theorems~\ref{thm:katetov-equivalent-I-ultrafilter-not-P-point}, \ref{thm:katetov-equivalent-I-ultrafilter-not-Q-point} and the fact that an ultrafilter is selective if and only if it is both a P-point and a Q-point.
\end{proof}


\section{By-products}
\label{sec:By-products}

In this section we collect some by-products of our studies.

\begin{proposition}
	Any $G_{\delta\sigma}$-ideal is extendable to a $F_\sigma$-ideal.
\end{proposition}

\begin{proof}
	Let $\J$ be a $G_{\delta\sigma}$-ideal. We will show that $\fin^2\not\preceq^{\omega_1}_{\omega^\omega,1-1}\J$. By Theorem~\ref{thm:new-order-above-FINxFIN-iff-extendable-to-Pplus}, this will finish the proof, because for Borel ideals extendability to $P^+$-ideals is equivalent to extendability to $F_\sigma$-ideals (see \cite[Theorem 3.2.7]{alcantara-phd-thesis}). We claim that $\J'=\J$ witnesses $\fin^2\not\preceq^{\omega_1}_{\omega^\omega,1-1}\J$. Let $\cA=\{A_i:\ i\in\omega\}\subseteq\cP(\omega)$ be such that $|\cA|=\omega$ and $\omega\notin\J(\cA)$. Observe that 
	$$\J(\cA)=\left\{B\subseteq\omega:\exists n\in\omega(B\setminus(A_0\cup\ldots\cup A_n)\in\J)\right\}=\bigcup_{n\in\omega}\phi_n^{-1}[\J],$$
	where $\phi_n:\cP(\omega)\to\cP(\omega)$ is given by $\phi_n(B)=B\setminus(A_0\cup\ldots\cup A_n)$. It is easy to check that each $\phi_n$ is continuous. Thus, $\J(\cA)$ is $G_{\delta\sigma}$. Since the ideal $\J(\cA)$ is $\Pi_4^0$, its rank is less than or equal to $1$ (by \cite[Theorem 9.1]{Debs}; see \cite[Definition 3.1]{Debs} for the definition of rank of an analytic ideal). Hence, using \cite[Theorem 7.5]{Debs} we get that $\fin^2\not\sqsubseteq\J(\cA)$. Thus, $\fin^2\not\leq_{\Kat}\J(\cA)$ by \cite[Example 4.1]{MR3034318}.
\end{proof}

A set $A\subseteq \omega$ is \emph{lacunary} (\cite{MR908265}) if $\lim_{n\to \infty}(a_{n+1}-a_n)=\infty$, where $(a_n)_{n\in \omega}$ is the increasing enumeration of all elements of $A$. The ideal $\lacunary$ generated by lacunary sets was examined for instance in \cite{flaskova-phd-thesis} (where the author used the name $SC$-set for a lacunary set) and 
\cite{MR3034318}. It is not difficult to show that $\lacunary$ is not a $P^+$-ideal (see \cite[Lemma~1.2.8]{flaskova-phd-thesis})

\begin{proposition}
	\label{prop:Lacunary-is-P+(L)}
	The ideal $\lacunary$ is a $P^+(\lacunary)$-ideal.
\end{proposition}

\begin{proof}
We will use the following characterization (\cite[Lemma~1.2.7]{flaskova-phd-thesis}):
	$$A\notin \lacunary \iff \forall_{n\in \omega} \exists_{d\in \omega}\forall_{K\in \omega}\exists_{k\geq K} \,(|A\cap [k,k+d]|\geq n).$$
	
	Let $A_n\notin \lacunary$ and $A_n\supseteq A_{n+1}$ for each $n\in \omega$. 
Using the above characterization for each set $A_n$ separately, we find $d_n\in \omega$ such that for each $K\in \omega$ there is $k\geqslant K$ with $|A_n\cap [k,k+d_n]|\geq n$.
Then, for each $n$, we choose an increasing  sequence $k^n_0<k^n_1<\dots$ such that 
$\{k^n_i:i\in\omega\}\in \lacunary$ 
and
$|A_n\cap [k^n_i,k^n_i+d_n]|\geq n$ for each $i$.
Finally, we define
$$A= \bigcup_{n\in\omega} \left(\bigcup_{i\in\omega}  A_n\cap [k^n_i,k^n_i+d_n]\right).$$

Using the above mentioned characterization, it is easy to see that $A\notin\lacunary$.

Below, we show that $A\setminus A_n\in \lacunary$ for each $n$.

First, we observe that for each $n$ the set 
$$B_n = \bigcup_{i\in\omega}  A_n\cap [k^n_i,k^n_i+d_n]$$
is a union of $d_n+1$ lacunary sets, so $B_n\in \lacunary$.
Then, taking into account that $A_0\supseteq A_1\supseteq \dots$, we notice that 
$A\setminus A_n \subseteq \bigcup_{i<n}B_i\in\lacunary$.
\end{proof}

\begin{proposition}
	The ideal $\lacunary$ is not $\leq_{\Kat}$-homogeneous (in particular, it is not homogeneous).
\end{proposition}

\begin{proof}
Since $\lacunary\subseteq\I_d$ (\cite[Lemma 1.4.4]{flaskova-phd-thesis} and $\fin^2\not\leq_{\Kat}\I_d$ (by Proposition~\ref{prop:P-ideal-not-K-above-FINxFIN}(\ref{prop:P-ideal-not-K-above-FINxFIN:item}) and the fact that $\I_d$ is a P-ideal), we obtain that $\fin^2\not\leq_{\Kat}\lacunary$. Moreover, by Proposition~\ref{prop:Lacunary-is-P+(L)}, $\lacunary$ is a $P^+(\lacunary)$-ideal.

Now, assume towards contradiction that  $\lacunary$ is $\leq_{\Kat}$-homogeneous. Then, using Theorem~\ref{thm:P-poiints-not-J-ultrafilter-equivalence}, under CH we obtain a P-point which is not an $\lacunary$-ultrafilter. However, this contradicts Proposition~\ref{prop:P-point-is-I-ultrafilter-for-non-finBW-ideals}.
\end{proof}

\begin{proposition}
	The ideal $\random$ generated by homogeneous sets in some fixed instance of the random graph on $\omega$ 
	is not $\leq_{\Kat}$-homogeneous (in particular, it is not homogeneous).
\end{proposition}

\begin{proof}
Note that 
$\ED\not\leq_{\Kat} \random$ (see \cite[Section 2]{MR3696069})
and $\random$ is a $P^+(\random)$-ideal (by Proposition~\ref{prop:examples-of-P-plus-ideals}(\ref{prop:examples-of-P-plus-ideals:F-sigma}) and the fact that $\random$ is an $F_{\sigma}$ ideal -- see \cite[page 30]{alcantara-phd-thesis}).

Now, assume towards contradiction that  $\random$ is $\leq_{\Kat}$-homogeneous. Then, using Theorem~\ref{prop:selective-ultrafilter-not-J-ultrafilter-equivalence}, under CH we would obtain 
a selective ultrafilters which is not an $\random$-ultrafilter. However, this contradicts Theorem~\ref{thm:known-ultrafilters-characterized-as-I-ultrafilters}(\ref{thm:known-ultrafilters-characterized-as-I-ultrafilters:Ramsey-point}).
\end{proof}


\bibliographystyle{amsplain}
\bibliography{paper}

\providecommand{\bysame}{\leavevmode\hbox to3em{\hrulefill}\thinspace}
\providecommand{\MR}{\relax\ifhmode\unskip\space\fi MR }
\providecommand{\MRhref}[2]{%
  \href{http://www.ams.org/mathscinet-getitem?mr=#1}{#2}
}
\providecommand{\href}[2]{#2}
\begin{thebibliography}{10}

\bibitem{MR2835960}
Pawe\l{} Barbarski, Rafa\l{} Filip\'{o}w, Nikodem Mro\.{z}ek, and Piotr Szuca,
  \emph{Uniform density {$u$} and {$\mathcal{I}_u$}-convergence on a big set},
  Math. Commun. \textbf{16} (2011), no.~1, 125--130. \MR{2835960}

\bibitem{MR3034318}
\bysame, \emph{When does the {K}at\v{e}tov order imply that one ideal extends
  the other?}, Colloq. Math. \textbf{130} (2013), no.~1, 91--102. \MR{3034318}

\bibitem{MR1350295}
Tomek Bartoszy\'{n}ski and Haim Judah, \emph{Set theory}, A K Peters, Ltd.,
  Wellesley, MA, 1995, On the structure of the real line. \MR{1350295}

\bibitem{MR1335140}
James~E. Baumgartner, \emph{Ultrafilters on {$\omega$}}, J. Symbolic Logic
  \textbf{60} (1995), no.~2, 624--639. \MR{1335140}

\bibitem{MR643555}
Murray~G. Bell, \emph{On the combinatorial principle {$P({\mathfrak{c}})$}},
  Fund. Math. \textbf{114} (1981), no.~2, 149--157. \MR{643555}

\bibitem{MR2768685}
A.~Blass, \emph{Combinatorial cardinal characteristics of the continuum},
  Handbook of set theory. {V}ols. 1, 2, 3, Springer, Dordrecht, 2010,
  pp.~395--489. \MR{2768685}

\bibitem{MR1845008}
A.~Blass, R.~Frankiewicz, G.~Plebanek, and C.~Ryll-Nardzewski, \emph{A note on
  extensions of asymptotic density}, Proc. Amer. Math. Soc. \textbf{129}
  (2001), no.~11, 3313--3320. \MR{1845008}

\bibitem{MR3395353}
Andreas Blass, Natasha Dobrinen, and Dilip Raghavan, \emph{The next best thing
  to a {P}-point}, J. Symb. Log. \textbf{80} (2015), no.~3, 866--900.
  \MR{3395353}

\bibitem{Brendle-slides}
J\"{o}rg Brendle, \emph{Cardinal invariants of analytic quotients}, Presented
  at ESI workshop on large cardinals and descriptive set theory in Vienna
  (\url{http://www.logic.univie.ac.at/2009/esi/pdf/brendle.pdf}), 2009.

\bibitem{MR3868039}
J\"{o}rg Brendle, Barnab\'{a}s Farkas, and Jonathan Verner, \emph{Towers in
  filters, cardinal invariants, and {L}uzin type families}, J. Symb. Log.
  \textbf{83} (2018), no.~3, 1013--1062. \MR{3868039}

\bibitem{MR3600759}
J\"{o}rg Brendle and Jana Fla\v{s}kov\'{a}, \emph{Generic existence of
  ultrafilters on the natural numbers}, Fund. Math. \textbf{236} (2017), no.~3,
  201--245. \MR{3600759}

\bibitem{MR908265}
T.~C. Brown and A.~R. Freedman, \emph{Arithmetic progressions in lacunary
  sets}, Rocky Mountain J. Math. \textbf{17} (1987), no.~3, 587--596.
  \MR{908265}

\bibitem{Debs}
Gabriel Debs and Jean Saint~Raymond, \emph{Filter descriptive classes of
  {B}orel functions}, Fund. Math. \textbf{204} (2009), no.~3, 189--213.
  \MR{2520152}

\bibitem{MR2320288}
Rafa\l{} Filip\'{o}w, Nikodem Mro\.{z}ek, Ireneusz Rec\l{}aw, and Piotr Szuca,
  \emph{Ideal convergence of bounded sequences}, J. Symbolic Logic \textbf{72}
  (2007), no.~2, 501--512. \MR{2320288}

\bibitem{MR2798896}
J.~Fla\v{s}kov\'{a}, \emph{{$\mathcal{I}$}-ultrafilters and summable ideals},
  10th {A}sian {L}ogic {C}onference, World Sci. Publ., Hackensack, NJ, 2010,
  pp.~113--123. \MR{2798896}

\bibitem{Flaskova-poster}
\bysame, \emph{Almost disjointness and its generalizations}, Presented at 6th
  European Congress of Mathematics
  (\url{{http://home.zcu.cz/~blobner/research/6ECMposter.pdf}}), 2012.

\bibitem{MR2194039}
Jana Fla\v{s}kov\'{a}, \emph{Thin ultrafilters}, Acta Univ. Carolin. Math.
  Phys. \textbf{46} (2005), no.~2, 13--19. \MR{2194039}

\bibitem{flaskova-phd-thesis}
\bysame, \emph{Ultrafilters and small sets}, Ph.D. thesis, Charles University
  in Prague, 2006.

\bibitem{MR2512901}
\bysame, \emph{A note on {$\mathcal{I}$}-ultrafilters and {P}-points}, Acta
  Univ. Carolin. Math. Phys. \textbf{48} (2007), no.~2, 43--48. \MR{2512901}

\bibitem{MR2752957}
\bysame, \emph{The relation of rapid ultrafilters and {$Q$}-points to van der
  {W}aerden ideal}, Acta Univ. Carolin. Math. Phys. \textbf{51} (2010),
  no.~suppl., 19--27. \MR{2752957}

\bibitem{MR1181163}
J.~A. Fridy, \emph{Statistical limit points}, Proc. Amer. Math. Soc.
  \textbf{118} (1993), no.~4, 1187--1192. \MR{1181163}

\bibitem{MR3692233}
M.~Hru\v{s}\'{a}k, D.~Meza-Alc\'{a}ntara, E.~Th\"{u}mmel, and C.~Uzc\'{a}tegui,
  \emph{Ramsey type properties of ideals}, Ann. Pure Appl. Logic \textbf{168}
  (2017), no.~11, 2022--2049. \MR{3692233}

\bibitem{MR2777744}
Michael Hru\v{s}\'{a}k, \emph{Combinatorics of filters and ideals}, Set theory
  and its applications, Contemp. Math., vol. 533, Amer. Math. Soc., Providence,
  RI, 2011, pp.~29--69. \MR{2777744}

\bibitem{MR3696069}
\bysame, \emph{Kat\v{e}tov order on {B}orel ideals}, Arch. Math. Logic
  \textbf{56} (2017), no.~7-8, 831--847. \MR{3696069}

\bibitem{MR2861027}
Michael Hru\v{s}\'{a}k and Jonathan~L. Verner, \emph{Adding ultrafilters by
  definable quotients}, Rend. Circ. Mat. Palermo (2) \textbf{60} (2011), no.~3,
  445--454. \MR{2861027}

\bibitem{MR1940513}
Thomas Jech, \emph{Set theory}, Springer Monographs in Mathematics,
  Springer-Verlag, Berlin, 2003, The third millennium edition, revised and
  expanded. \MR{1940513}

\bibitem{MR1321597}
Alexander~S. Kechris, \emph{Classical descriptive set theory}, Graduate Texts
  in Mathematics, vol. 156, Springer-Verlag, New York, 1995. \MR{1321597}

\bibitem{MR3640040}
P.~Klinga and A.~Nowik, \emph{Extendability to summable ideals}, Acta Math.
  Hungar. \textbf{152} (2017), no.~1, 150--160. \MR{3640040}

\bibitem{MR3594409}
A.~Kwela and J.~Tryba, \emph{Homogeneous ideals on countable sets}, Acta Math.
  Hungar. \textbf{151} (2017), no.~1, 139--161. \MR{3594409}

\bibitem{Kwela-question-of-Hrusak}
Adam Kwela, \emph{On extendability to ${F}_\sigma$-ideals}, manuscript
  (\url{https://arxiv.org/abs/2106.13053}) (2021).

\bibitem{Unboring}
\bysame, \emph{Unboring ideals}, manuscript
  (\url{https://arxiv.org/abs/2103.17166}) (2021).

\bibitem{MR1124539}
Krzysztof Mazur, \emph{{$F_\sigma$}-ideals and {$\omega_1\omega_1^*$}-gaps in
  the {B}oolean algebras {$P(\omega)/I$}}, Fund. Math. \textbf{138} (1991),
  no.~2, 103--111. \MR{1124539}

\bibitem{alcantara-phd-thesis}
D.~Meza-Alc\'{a}ntara, \emph{Ideals and filters on countable set}, Ph.D.
  thesis, Universidad Nacional Aut\'{o}noma de M\'{e}xico, 2009.

\bibitem{MR1708146}
S\l{}awomir Solecki, \emph{Analytic ideals and their applications}, Ann. Pure
  Appl. Logic \textbf{99} (1999), no.~1-3, 51--72. \MR{1708146}

\bibitem{MR698462}
A.~Szyma\'{n}ski and Zhou~Hao Xua, \emph{The behaviour of {$\omega ^{2^{\ast}
  }$} under some consequences of {M}artin's axiom}, General topology and its
  relations to modern analysis and algebra, {V} ({P}rague, 1981), Sigma Ser.
  Pure Math., vol.~3, Heldermann, Berlin, 1983, pp.~577--584. \MR{698462}

\end{thebibliography}

\end{document}